\documentclass[11pt]{article}
\usepackage{graphicx, caption, subcaption} % Required for inserting images

\usepackage{cite}

\usepackage{amsmath, amssymb, hyperref, multirow, longtable, paracol, booktabs, diagbox, array, makecell}
\usepackage{amsthm}
\usepackage{dsfont}
\usepackage[papersize={8.5in,11in}]{geometry}
\usepackage[normalem]{ulem}
\usepackage{xcolor}
\batchmode
\usepackage{mathtools}
\usepackage{mathrsfs,bigints,mathtools,graphicx,soul}
\usepackage{amsmath}
\usepackage[toc,page]{appendix}
\usepackage[affil-it]{authblk}
\usepackage{cite}

\DeclarePairedDelimiter{\abs}{\lvert}{\rvert}
\newtheorem{theorem}{Theorem}
\newtheorem{assumption}[theorem]{Assumption}
\newtheorem{remark}{Remark}
\newtheorem{lemma}{Lemma}
\newtheorem{proposition}{Proposition}

\newtheorem{definition}{Definition}

\title{\vspace{-2cm}Dimensional reduction of dynamical systems on graphons}

\author[1]{Bisna Mary Eldo}
\author[2]{Sarbendu Rakshit}
\author[1,3]{Naoki Masuda}

\affil[1]{†Department of Mathematics, University of Michigan, Ann Arbor, 48109-1043, MI, USA.}
\affil[2]{Department of Sciences and Humanities (Stream of Mathematics), Indian Institute of Information Technology Design and Manufacturing, Kancheepuram, Chennai, Tamil Nadu 600127, India}
\affil[3]{§Gilbert S. Omenn Department of Computational Medicine and Bioinformatics, University of Michigan, Ann Arbor, 48109-2218, MI, USA}
\date{}
\begin{document}
	\maketitle
	\begin{abstract}
		Dynamical systems on networks are inherently high-dimensional unless the number of nodes is extremely small. Dimension reduction methods for dynamical systems on networks aim to find a substantially lower-dimensional system that preserves key properties of the original dynamics such as bifurcation structure. A class of such methods proposed in network science research entails finding a one- (or low-) dimensional system that a particular weighted average of the state variables of all nodes in the network approximately obeys. We formulate and mathematically analyze this dimension reduction technique for dynamical systems on dense graphons, or the limiting, infinite-dimensional object of a sequence of graphs with an increasing number of nodes. We first theoretically justify the continuum limit for a nonlinear dynamical system of our interest, and the existence and uniqueness of the solution of graphon dynamical systems. We then derive the reduced one-dimensional system on graphons and prove its convergence properties. Finally, we perform numerical simulations for various graphons and dynamical system models to assess the accuracy of the one-dimensional approximation.
	\end{abstract}
	
	\section{Introduction}
	
	Many complex systems are dynamical and networked in nature, with applications ranging from modeling and inference of epidemic spreading, collective oscillations and synchronization in biological systems, and cascading failures in infrastructure systems, to climate dynamics. Dynamical systems on networks are by definition high-dimensional unless there are only a small number of nodes, because each node of the network usually carries at least one state variable. A strategy to facilitate understanding such high-dimensional dynamics on networks is dimension reduction. Various dimension reduction techniques have been developed to reduce high-dimensional data in general, aiming to map the given data or system to a lower-dimensional space without losing much information \cite{hinton2006reducing, cunningham2015linear}. Dimension reduction methods tailored for dynamical systems have also been developed. These techniques aim to approximate dynamical systems in higher-dimensional space by an alternative, more tractable lower-dimensional dynamical systems with an acceptable accuracy \cite{ghadami2022data, lu2021review, rega2005dimension, hesthaven2022reduced, rowley2017model}. These methods ideally help us extract key properties of the network dynamics such as invariants, modular structure, and bifurcations.
	
	In this article, we consider two closely related dimension reduction techniques for dynamical systems on networks proposed in network science research. One is to approximate the original dynamical system by the dynamics of a single observable that is a weighted average of the state variables of all nodes in the network, proposed by Gao, Barzel, and Barab\'{a}si \cite{gao2016universal}, which we refer to as GBB reduction. The other method is a generalization of the GBB reduction, which is a type of mean-field theory, to the case of general adjacency matrices of the network \cite{laurence2019spectral}, which we refer to as spectral reduction (also called DART \cite{thibeault2020threefold}). The spectral reduction uses the eigenvalues and eigenvectors of the adjacency matrix to reduce the original dynamical system on networks to a lower-dimensional system in terms of weighted linear combinations of the state variable that are distinct from what the GBB reduction uses.
	
	We mathematically study the GBB and spectral reductions on large networks using graphons. A graphon is a continuum limit of graphs (i.e., networks) as the number of nodes tends to infinity \cite{lovasz2006limits,borgs2008convergent,borgs2018p,lovasz2012large}. Graphons are represented by a symmetric measurable function $W:[0,1]\times[0,1] \to [0,1]$, where $W(y, y')$ dictates the probability that two nodes parameterized by $y$ and $y'$ are adjacent to each other. Graphons are an emerging topic at the intersection of mathematics, statistics, network science, and some other fields. They have been used for modeling power network dynamics \cite{kuehn2019power}, epidemics spreading \cite{vizuete2020graphon}, synchronization of oscillators \cite{medvedev2014nonlinear, medvedev2020kuramoto}, random walks \cite{petit2021random}, and consensus formation \cite{bonnet2022consensus}, for example. The continuum limit of graph sequences provides a framework for the study of dynamical systems on massive graphs, where classical methods would become computationally intractable \cite{petit2021random}. In this article, we focus on the continuum limit of a class of nonlinear dynamical systems and its dimension reduction.
	
	The organization of this paper is as follows. In sections \ref{sec_2} and \ref{sec_3}, we study mathematical properties of graphon dynamical systems such as the existence, uniqueness and the boundedness of the solution. We also prove the convergence of the solution for the finite network to the corresponding continuum limit solution. These parts are mainly extensions of Ref.~\cite{medvedev2014nonlinear}, which dealt with a different class of graphon dynamical systems. 
	In section \ref{sec_4}, we define and theoretically and numerically study the GBB and spectral reductions of graphon dynamical systems.
	
	\section{Dynamics on finite networks and graphon dynamical systems}\label{sec_2}
	
	Various complex systems in the real world can be described by dynamical systems on networks~\cite{boccaletti2006complex, barrat2008dynamical, porter2016dynamical}. Dynamical systems on networks are necessarily high-dimensional because each node is assigned with one or more dynamical variables and these variables interact via edges of the given network. Even if the dynamics at each node is one-dimensional, the entire dynamical system on the network is $N$-dimensional, where $N$ is the number of nodes in the network.
	
	We consider the class of dynamical systems on networks given by
	\begin{equation}\label{eq_1}
		\dfrac{d x_i}{dt}=f(x_i)+\frac{1}{N}\sum\limits_{j=1}^NA_{ij}G(x_i,x_j),
	\end{equation}
	where $x_i(t) \in \mathbb{R}$ represents the dynamical state of the $i$th node at time $t$, $f(x_i) \in \mathbb{R}$ represents the single-node dynamics, $G(x_i,x_j) \in \mathbb{R}$ describes the interaction between the $i$th and $j$th nodes, and $A_{ij}$ $(\ge 0)$ is the strength of the influence of node $j$ on node $i$, corresponding to the weighted adjacency matrix of the network. Equation \eqref{eq_1} has been used for studying dimension reduction of dynamics on networks in network science research. They are the Gao-Barzel-Barab\'{a}si (GBB) reduction \cite{gao2016universal} and the spectral reduction \cite{laurence2019spectral, thibeault2020threefold}, including their various extensions \cite{tu2017collapse, kundu2022accuracy, masuda2022dimension, vegue2023dimension, tu2021dimensionality, salgado2024dimension}. Equation~\eqref{eq_1} is also used for studying the interplay between the structure and the dynamics of complex systems \cite{barzel2013universality}, early warning signals \cite{maclaren2023early,masuda2024anticipating, maclaren2024applicability}, signal propagation in complex networks \cite{hens2019spatiotemporal}, interplay between network motifs and dynamical stability \cite{gross2023dense}, and information flow \cite{harush2017dynamic}. Motivated by these previous studies, we study the continuum limit of this family of dynamical systems. 
	
	To state graphon dynamical systems corresponding to \eqref{eq_1},
	we start by introducing graphons. A graphon is a continuum limit of a growing sequence of finite networks (i.e., graphs). The growth of networks means that $N$ increases in sequence. In the graphon formalism, the nodes are placed in $I = [0,1]$. In intuitive terms, two nodes with latent variables $y, y' \in I$ are adjacent by an edge with probability $W(y, y')$. Every Cauchy sequence of graphs with an increasing $N$ has a limit object under the cut norm, known as graphon~\cite{lovasz2006limits, borgs2018p, borgs2019𝐿,petit2021random}. Precisely speaking, a graphon is a symmetric measurable function $W:[0,1]^2 \to [0,1]$. 
	The norm on $W$ is given by the usual $L^p$ norm and is denoted by $\left\|W\right\|_{p}$. See Refs.~\cite{lovasz2006limits,kaliuzhnyi2017semilinear,petit2021random,medvedev2014nonlinear} for theoretical frameworks of graphons as a limiting object of dense and sparse graph sequences.
	
	We will prove the convergence of \eqref{eq_1} and its one-dimensional approximation to the corresponding graphon dynamical systems as $N\to\infty$. To this end, we use step graphons \cite{petit2021random}. Consider a dynamical system on a finite network given by \eqref{eq_1}. The step graphon is a method to formulate an arbitrary finite network by graphons. Consider a uniform partition $\mathcal{P}_{N} = \{ I_{N,i} \}_{i=1, \ldots, N}$ of $I$, where $I_{N,i} = \left[\frac{i-1}{N}, \frac{i}{N}\right)$ for $i\in \{1, \ldots, N-1\}$, and $I_{N,N} = \left[\frac{N-1}{N}, 1 \right]$. Each element in $\mathcal{P}_{N}$ represents a node in a network having $N$ nodes. Let us generate the step graphon version of the adjacency matrix $A$, and then \eqref{eq_1}. As in Ref.~\cite{petit2021random}, we define a map $\eta:\mathbb{M}_N \to \mathbb{W}$, where $\mathbb{M}_N$ is the space of $N \times N$ matrices of which all entries are in $[0, 1]$ \cite{lovasz2012large, tripathi2024stochastic, borgs2008convergent}, and $\mathbb{W}$ is the space of all graphons. The map represents the step graphon and is defined by $\eta(A)(y,y')=\sum_{i=1}^N \sum_{j=1}^N A_{ij}\chi_{I_{N,i}}(y)\chi_{I_{N,j}}(y')$, where $A$ is the adjacency matrix with entries in $\left[0,1\right]$, and $\chi_{I_{N,i}}$ is the indicator function, i.e.,
	%
	% \bme{As we discussed on 11/26/2024 I have added one more reference. In \cite{lovasz2012large} after corollary 9.13 says about the bound of each entries of the matrix}
	\begin{equation}\label{eq_a1}
		\chi_{I_{N,i}}(y) =
		\begin{cases}
			1 & \text{if } y\in I_{N,i},\\
			0 & \text{if } y\not\in I_{N,i}.
		\end{cases}
	\end{equation}
	
	To write \eqref{eq_1} in terms of the step graphon, we consider any vector ${\bf u}=(u_1, \ldots, u_N) \in \mathbb{R}^N$ and define a map $\tilde{\eta}$ from ${\bf u}$ to a piecewise constant function on $[0,1]$ by
	\begin{equation}\label{eq_7}
		\tilde{\eta}({\bf u})(y):=\sum_{i=1}^N u_i\chi_{I_{N,i}}(y).
	\end{equation}
	Using \eqref{eq_7}, we obtain
	\begin{equation}\label{eq_8}
		\tilde{\eta}(f(\pmb{x}))(y)=\sum\limits_{i=1}^N f(x_i)\chi_{I_{N,i}}(y),
	\end{equation}
	where $\pmb{x}= (x_1, \ldots, x_N)^\top$, $f(\pmb{x}) = \left( f(x_1), \ldots,  f(x_N)\right)^\top$, and ${}^{\top}$ represents the transposition.
	
	For any $y \in [0, 1]$, we obtain $y\in I_{N, \ell}$ for a particular $\ell$. This implies that $\chi_{I_{N, i}}(y)=1$ if $i=\ell$, and $\chi_{I_{N,i}}(y)=0$ if $i\ne \ell$. Using this $\ell$, one can rewrite \eqref{eq_8} as $\tilde{\eta}(f(\pmb{x}))(y)=f(x_{\ell})$, which yields $\tilde{\eta}(f(\pmb{x}))(y)=f(\tilde{\eta}(\pmb{x})(y))$. 
	Using this relationship and the unique values of $i$ and $j$ satisfying $y \in I_{N, i}$ and $y' \in I_{N, j}$ for any given $y, y' \in [0, 1]$, we obtain from \eqref{eq_1}
	\begin{equation}
		\dfrac{\partial\,x(y,t)}{\partial t}=f(x(y,t))+\int_{I}\eta(A)(y,y')\,G(x(y,t),x(y',t)) \mathrm{d}y'.
		\label{eq:pre-graphon-dynamical-system}
	\end{equation}
	Note that we use $x$ to represent the dynamical state of the nodes in the cases of both finite discrete networks and graphons. This convention should not arise confusion because we use subscripts in the former case (e.g., $x_i$) and an argument taking the value in a unit interval (e.g., $x(y)$, where $y\in [0, 1]$) in the latter case.
	
	Now we introduce a general graphon dynamical system (see e.g.\,Refs.~\cite{gao2018graphon, gao2019graphon, petit2021random}). We define
	\begin{equation}\label{eq4}
		\frac{\partial\,x(y,t)}{\partial t}=f(x(y,t))+\int_0^1W(y,y')\,G(x(y,t),x(y',t))\,\mathrm{d}y',
	\end{equation}
	where $W$ in \eqref{eq4} is any graphon. Equation~\eqref{eq4} is different from~\eqref{eq:pre-graphon-dynamical-system}, in which $\eta(A)$ is confined to be the step graphon induced by an $N\times N$ adjacency matrix.
	
	Conversely, the discrete-network variant of the graphon dynamical system derived from \eqref{eq4} is given by
	\eqref{eq_1} \cite{borgs2008convergent, petit2021random}, or equivalently \eqref{eq:pre-graphon-dynamical-system}, where
	\begin{equation}
		A_{ij} = (W_N)_{ij} \equiv N^2\int_{I_{N,i}\times I_{N,j}}W(y,y')\mathrm{d}y\mathrm{d}y' \in [0, 1].
		\label{eq:Aij-from-W}
	\end{equation}
	
	We have two main mathematical goals with this article. First, we prove that the solution of the discrete dynamical system converges to the graphon dynamical system as $N \to \infty$, i.e., \eqref{eq:pre-graphon-dynamical-system} combined with \eqref{eq:Aij-from-W} converges to \eqref{eq4} as $N \to \infty$. We achieve this by mainly using the results in a previous study \cite{medvedev2014nonlinear}. Second, we develop graphon versions of the GBB and spectral reduction methods, approximating the solutions of \eqref{eq4}, and prove the convergence for the GBB reduction. Then, we numerically examine the accuracy of the two dimension reduction methods for graphon dynamical systems.
	
	\section{Continuum limit}\label{sec_3}
	
	In this section, we mathematically establish the graphon continuum limit of 
	\eqref{eq_1}. The formalism and results in this section are largely similar to those in Ref.~\cite{kaliuzhnyi2017semilinear}, which established such a limit and proved various properties of the limit. Therefore, we explicitly mention how each of our formalisms and proofs are similar to or different from theirs. In Ref.~\cite{kaliuzhnyi2017semilinear}, sparse networks induced by the graphon are considered. However, we focus on dense networks in this paper. This is because $f$ does not generally commute with $\mathbb{E}$, the expectation with respect to the random ensemble of sparse graphs. In addition, $\mathbb{E}[A_{ij} G (x_i, x_j)]$, where $A_{ij}$ is a random variable, and therefore $x_i$ and $x_j$ are as well, is not equal to $\mathbb{E}[A_{ij}] \mathbb{E}[G(x_i, x_j)]$ in general due to correlation between $A_{ij}$ and $\{x_1, \ldots, x_N\}$, which may cause difficulty. We assume the following properties \cite{kaliuzhnyi2017semilinear}:
	\begin{assumption}[Lipschitz continuity]\label{assumption_1}
		First, the self-dynamics $f: \Omega \to \mathbb{R}$ is a Lipschitz continuous function on the closed interval $\Omega \subset \mathbb{R}$ with Lipschitz constant $L_f$, i.e., $|f(x) - f(x')| \leq L_f |x-x'|$ for all $ x, x' \in \Omega$.
		Second, the coupling function $G: \Omega \times \Omega \to \mathbb{R}$ is a Lipschitz continuous function with Lipschitz constant $L_G$, i.e., $\left|G(x_i, x_j) - G(x_i', x_j')\right| \leq L_G\left\|(x_i, x_j) - (x_i', x_j')\right\|_{2}$ for all $(x_i, x_j), ~ (x_i', x_j') \in \Omega \times \Omega$.
	\end{assumption}
	%\assump
	\begin{assumption}[Positively invariant domain]\label{assumption_2}
		The dynamical system \eqref{eq4} admits an invariant set $S \subset L^q(I)$ defined by $S = \{\pmb{x} \in L^q(I) | x(y, t) \in \Omega~ \text{for almost every} ~ y \in I\}$. That is, if $\pmb{x}(0) \in S$, then $\pmb{x}(t) \in S$ for all $t \geq 0$.
	\end{assumption}
	\begin{assumption}[W-1]\label{assumption_3}
		Let $W: I^2 \to I$ is a continuous function almost everywhere on $I^2$. 
	\end{assumption}
	
	\subsection{Existence, uniqueness, and boundedness of the solution}
	
	\begin{theorem}[Existence and uniqueness of the solution]\label{theorem_1}
		Suppose that $W \in L^p(I^2)$ with $p \geq 2$. Then, the initial value problem (IVP) for \eqref{eq4} with initial condition $\pmb{x}(0) = \mathbf{g}\in S \subset L^{q}(I), ~ q = \frac{p}{p-1}$ has a unique solution $\pmb{x}\in C^1(L^q(I); [0, T])$ that satisfies $\pmb{x}(\cdot,t)\in S$ for all $t\geq 0$ and continuously depends on $\mathbf{g}$. 
		%Let $\Omega \subset \mathbb{R}$ be any bounded interval. Assume that $f: \Omega \to \mathbb{R}$ and $G: \Omega \times \Omega \to \mathbb{R}$ satisfy Assumptions \ref{assumption_1} and \ref{assumption_2}. Let $g \in L^q(I; \mathbb{R})$ satisfies $g(y) \in \Omega$ for almost every $y \in I$ . Assume further that $\Omega$ is positively invariant for the graphon dynamical system. Then, for any $T>0$, the initial value problem has a unique solution $x \in C^1([0,T];L^q(I; \mathbb{R}))$. Moreover, the solution satisfies $x(y,t)\in \Omega$ for almost every $y \in I$ and all $t\in [0,T]$. The solution depends continuously on the initial condition.
	\end{theorem}
	
	This theorem is an adaptation of Theorem 3.1 in Ref.~\cite{kaliuzhnyi2017semilinear}; the following proof is similar to theirs. However,
	in contrast to their theorem and proof, we do not assume sparsity of the network, and we consider the Lipschitz continuity in a bounded domain.
	
	Since $f$ and $G$ are Lipschitz on the bounded sets $\Omega$ and $\Omega \times \Omega$, respectively, we extend them to globally Lipschitz functions on $\mathbb{R}$ and $\mathbb{R}^2$ and still denote the them by $f$ and $G$ for the purpose of the proof. We first solve the extended system in $C(L^q(I); [0, \tau])$ by the contraction mapping theorem. Assumption \ref{assumption_2} implies that, if $\pmb{g} \in S$, then the corresponding solution satisfies $\pmb{x}(\cdot, t) \in S$ for all $t \geq 0$. Therefore, the solution of the extended system coincides with the solution of the original system on $S$.
	
	\begin{proof}[Proof of Theorem \ref{theorem_1}]
		Let $L=\max\{L_f,L_G\}$ and $\mathcal{M}=C(L^q(I; \mathbb{R}); [0, \tau])$. Define $K:\mathcal{M}\to \mathcal{M}$ by
		\begin{equation}\label{eq_b}
			[K \pmb{x}](t)=\mathbf{g}+\int_0^t\left[ f(x(y,s)) + \int_IW(\cdot,y')G(x(y,s), x(y',s))\,\mathrm{d}y' \right]\,\mathrm{d}s.
		\end{equation}
		We rewrite the IVP \eqref{eq4} as a fixed point equation for the mapping $K$ as
		\begin{equation}\label{eq_a}
			\pmb{x}=K\,\pmb{x},
		\end{equation}
		and show that $K$ is a contraction on $\mathcal{M}$.
		We know that both $I^2$ and $I$ are of finite measure, and that $W$ is a measurable function from $I^2$ to $I$. Therefore, for any $x\in L^q(I)$ and $W\in L^p(I^2)$, where $p>1$ and $q=\frac{p}{p-1}$, we obtain
		\begin{equation}\label{eq_38}
			\left\|x\right\|_{L^q(I)}\le\left\|x\right\|_{L^{p\lor q}(I)}~\mbox{and}~~\left\|W\right\|_{L^p(I^2)}\le\left\|W\right\|_{L^{p\lor q}(I^2)},
		\end{equation}
		where $p\lor q$ represents the larger of $p$ and $q$.
		%For any $\pmb{x}_1,\pmb{x}_2\in \mathcal{M}$, 
		%using short-hand notations ${\bf x'} \equiv (x_1(\cdot,t), x_2(\cdot,t))$ and
		%${\bf y'} \equiv (x_1(y,t), x_2(y,t))$, and 
		By letting $L = \max \{{L_f, L_G }\}$ and
		\begin{equation}\label{eq:norm}
			\left\|K\pmb{x}\right\|_\mathcal{M} = \max\limits_{t\in[0,\tau]} \, \left\|K x(\cdot,t)\right\|_{L^q(I)},
		\end{equation}
		we obtain
		\begin{align}\label{eq_4d}
			& \left\|K\pmb{x}_1-K\pmb{x}_2\right\|_\mathcal{M}\nonumber\\
			=& \max\limits_{t\in[0,\tau]}\left\|Kx_1(t)-Kx_2(t)\right\|_{L^q(I)}\nonumber\\
			\leq& \max\limits_{t\in[0,\tau]}\left\|\int_0^t\left[L\,\abs{x_1(\cdot,s)-x_2(\cdot,s)}+\int_IW(\cdot, y')\abs{G(x_1(\cdot,s), x_1(y',s))-G(x_2(\cdot,s), x_2(y',s))}\,\mathrm{d}y'\right]\mathrm{d}s\right\|_{L^q(I)} \nonumber \\
			\leq& L\max\limits_{t\in[0,\tau]}\left\|\int_0^t\left[\abs{x_1(\cdot,s)-x_2(\cdot,s)}+\int_IW(\cdot,y')\left\|\left(x_1(\cdot,s), x_1(y',s)\right) - \left(x_2(\cdot,s), x_2(y',s)\right)\right\|_{2}\,\mathrm{d}y'\right]\mathrm{d}s\right\|_{L^q(I)}\nonumber\\
			\leq& L\max\limits_{t\in[0,\tau]}\left\{\left\|x_1(\cdot,t)-x_2(\cdot,t)\right\|_{L^q(I)}+\left\|\int_IW(\cdot,y')\abs{x_1(\cdot,t)-x_2(\cdot,t)}\,\mathrm{d}y'\right\|_{L^q(I)} \right. \nonumber\\
			& \left. + \left\|\int_IW(\cdot,y')\abs{x_1(y',t)-x_2(y',t)}\,\mathrm{d}y'\right\|_{L^q(I)}\right\}.
			%&\le \tau %%L\max\limits_{t\in[0,\tau]}\bigg\{2\norm{\omega_1(\cdot,t)-\omega_2(\cdot,t)}_{L^q(I)}+\norm{\bigintss_IW(\cdot,y)\abs{\omega_1(x,t)-\omega_2(x,t)}\,dy}_{L^q(I)}\bigg\}.
		\end{align}
		%\begin{equation}
		%\norm{K\omega_1-K\omega_2}_\mathcal{M}&=\max\limits_{t\in[0,\tau]}\norm{K\omega_1(t)-K\omega_2(t)}_{L^q(I)}\\[8pt]
		%&\leq \max\limits_{t\in[0,\tau]}\norm{\int_0^t\Big[L\,\abs{\omega_1(\cdot,s)-\omega_2(\cdot,s)}+\int_IW(\cdot,y)\norm{G({\bf x})-G({\bf y})}\,dy\Big]ds}_{L^q(I)} \\[8pt]
		%&\text{where}~{\bf x} = (\omega_1(\cdot,t), \omega_2(\cdot,t)), {\bf x} = (\omega_1(y,t), \omega_2(y,t))\\[8pt] 
		%&\leq L\max\limits_{t\in[0,\tau]}\norm{\int_0^t\Big[\abs{\omega_1(\cdot,s)-\omega_2(\cdot,s)}+\int_IW(\cdot,y)\norm{{\bf x}-{\bf y}}\,dy\Big]ds}_{L^q(I)}\\[8pt]
		%&\leq L\max\limits_{t\in[0,\tau]}\bigg\{\norm{\omega_1(\cdot,t)-\omega_2(\cdot,t)}_{L^q(I)}+\norm{\int_IW(\cdot,y)\abs{\omega_1(\cdot,t)-\omega_2(\cdot,t)}\,dy}_{L^q(I)}\\&~~~+\norm{\int_IW(\cdot,y)\abs{\omega_1(y,t)-\omega_2(y,t)}\,dy}_{L^q(I)}\bigg\}\\[8pt]
		%&\le \tau L\max\limits_{t\in[0,\tau]}\bigg\{2\norm{\omega_1(\cdot,t)-\omega_2(\cdot,t)}_{L^q(I)}+\norm{\bigintss_IW(\cdot,y)\abs{\omega_1(x,t)-\omega_2(x,t)}\,dy}_{L^q(I)}\bigg\}.
		%\end{equation}
		We evaluate the last term on the right-hand side of \eqref{eq_4d} as follows:
		\begin{align}\label{eq_4e}
			%\begin{array}{lllll}
			\left\|\int_IW(\cdot,y')\abs{x_1(y',t)-x_2(y',t)}\,\mathrm{d}y'\right\|_{L^q(I)}&\leq \left\|\int_IW(\cdot,y')\mathrm{d}y' \, \left\|x_{1}(\cdot,t)  - x_{2}(\cdot,t)\right\|_{L^q(I)}\right\|_{L^q(I)}\nonumber\\
			& = \left\|\int_IW(y',\cdot)\mathrm{d}y' \, \left\|x_{1}(\cdot,t)  - x_{2}(\cdot,t)\right\|_{L^q(I)}\right\|_{L^q(I)}\nonumber\\
			& = \left\|\int_IW(y,\cdot)\mathrm{d}y \, \left\|x_{1}(\cdot,t)  - x_{2}(\cdot,t)\right\|_{L^q(I)}\right\|_{L^q(I)}\nonumber\\
			&\le \left\|\,\left\|W(y,\cdot)\right\|_{L^p(I)}\, \cdot \left\|x_1(\cdot,t)-x_2(\cdot,t)\right\|_{L^q(I)}\right\|_{L^q(I)}\nonumber\\
			&\le \left\|\,\left\|W(y,\cdot)\right\|_{L^p(I)}\right\|_{L^q(I)}\, \cdot \left\|\, \left\|x_1(\cdot,t)-x_2(\cdot,t)\right\|_{L^q(I)}\right\|_{L^q(I)}\nonumber\\
			&\le\left\|W\right\|_{L^p(I^2)}\, \cdot \left\|\pmb{x}_1(t)-\pmb{x}_2(t)\right\|_{L^q(I)} \nonumber \\ 
			&\leq \left\|\pmb{x}_1(t)-\pmb{x}_2(t)\right\|_{L^q(I)}
			%\end{array}
		\end{align}
		%\begin{equation}
		%\norm{\int_IW(\cdot,y)\abs{\omega_1(y,t)-\omega_2(y,t)}\,dy}_{L^q(I)}&\le \norm{\int_IW(\cdot,y)dy \, \norm{w_{1}(\cdot,t)  - w_{2}(\cdot,t)}_{L^q(I)}}_{L^q(I)}\\[8pt]
		%& = \norm{\int_IW(y,\cdot)dy \, \norm{w_{1}(\cdot,t)  - w_{2}(\cdot,t)}_{L^q(I)}}_{L^q(I)}\\[8pt]
		%& = \norm{\int_IW(x,\cdot)dx \, \norm{w_{1}(\cdot,t)  - w_{2}(\cdot,t)}_{L^q(I)}}_{L^q(I)}\\[8pt]
		%&\le\norm{\,\norm{W(x,\cdot)}_{L^p(I)}\,\norm{\omega_1(\cdot,t)-\omega_2(\cdot,t)}_{L^q(I)}}_{L^q(I)}\\[8pt]
		%&\le\norm{\,\norm{W(x,\cdot)}_{L^p(I)}}_{L^q(I)}\,\norm{\, \norm{\omega_1(\cdot,t)-\omega_2(\cdot,t)}_{L^q(I)}}_{L^q(I)}\\[8pt]
		%&\le\norm{W}_{L^p(I^2)}\,\norm{\omega_1(t)-\omega_2(t)}_{L^q(I)}
		%\end{equation}
		Similarly, we obtain
		\begin{equation}\label{eq_4f}
			\left\|\int_IW(\cdot, y')\left|x_1(\cdot,t)-x_2(\cdot,t)\right|\,\mathrm{d}y'\right\|_{L^q(I)} \leq \left\|W\right\|_{L^p(I^2)}\,\left\|\pmb{x}_1(t)- \pmb{x}_2(t)\right\|_{L^q(I)} \leq \left\|\pmb{x}_1(t)- \pmb{x}_2(t)\right\|_{L^q(I)} .
		\end{equation}
		By substituting \eqref{eq_4e} and \eqref{eq_4f} in \eqref{eq_4d}, we obtain
		\begin{equation}
			\left\|K\pmb{x}_1-K\pmb{x}_2\right\|_\mathcal{M}\le 3L\tau\left\|\pmb{x_1}- \pmb{x_2}\right\|_\mathcal{M}.
		\end{equation}
		By setting $\tau=\dfrac{1}{2L\left(2\left\|W\right\|_{L^p(I^2)}+1\right)}$, we obtain $\left\|K\pmb{x}_1-K\pmb{x}_2\right\|_\mathcal{M}\le \frac{1}{2}\left\|\pmb{x_1}- \pmb{x_2}\right\|_\mathcal{M}$. Therefore, $K$ is a contraction on $\mathcal{M}$.
		
		Finally, to show that $K(\mathcal{M})\subset \mathcal{M}$, we define $z(y, t) := 0$, $\forall (y, t) \in I\times[0,\tau]$. We obtain
		\begin{align}\label{eq_used-for-th3.8}
			\left\|K\pmb{x}\right\|_\mathcal{M}&\le\left\|K\pmb{x}-K\pmb{z}\right\|_\mathcal{M}+\left\|K\pmb{z}\right\|_\mathcal{M}\nonumber\\
			&\le\frac{1}{2}\left\|\pmb{x}-\pmb{z}\right\|_\mathcal{M}+\left\|K\pmb{z}\right\|_\mathcal{M}\nonumber\\
			&=\frac{1}{2}\left\|\pmb{x}\right\|_\mathcal{M}+\left\|K\pmb{z}\right\|_\mathcal{M}.
			%	\label{eq:Kx-bound-by-Kz}
		\end{align}
		By substituting $z(y, t)=c$, where $c \in \Omega$ in \eqref{eq_b}, we obtain
		\begin{equation}
			[K \pmb{z}](t) = \mathbf{g}+t(f(c) + s(y) G(c, c)),
			\label{eq:Kz-eval}
		\end{equation}
		where 		
		\begin{equation}
			s(y) := \int_0^1 W(y, y') \mathrm{d}y'.
			\label{eq:def-s(y)}
		\end{equation}
		Because $0\leq s(y) \leq 1$, it is guaranteed that $K\pmb{z} \in {\mathcal{M}}$. Combination of \eqref{eq_used-for-th3.8} and $K\pmb{z} \in {\mathcal{M}}$ implies that $K\pmb{x}\in \mathcal{M}$, i.e., $K(\mathcal{M})\subset \mathcal{M}$.
		
		Because $K$ is contracting and $K(\mathcal{M})\subset \mathcal{M}$, by the Banach contraction mapping principle, there exists a unique solution for the IVP \eqref{eq4}, $\pmb{x}\in \mathcal{M}\subset C(L^q(I); [0,\tau])$.
		
		Finally, since the initial condition satisfies $g(y) \in \Omega$ for almost every $y \in I$, we have $\pmb{g} \in S$. By Assumption \ref{assumption_2}, the corresponding solution satisfies $\pmb{x}(\cdot, t) \in S$ for all $t \geq 0$. Equivalently, $x(y, t) \in \Omega$ for almost every $y \in I$ and all $t \geq 0$. Therefore, although the proof uses globally Lipschitz extensions of $f$ and $G$, the solution remains in $S$. Hence, along this solution, the extended system coincides with the original graphon dynamical system.
	\end{proof}
	
	\begin{remark}\label{corol_1}
		Let $\Omega = [a,b] \subset \mathbb{R}$ be the bounded interval associated with the invariant domain $S \subset L^q(I)$ for the graphon dynamical system. Suppose that the initial condition satisfies $g(y) \in \Omega$ for almost every $y \in I$. Then, the solution obtained in Theorem \ref{theorem_1} is uniformly bounded on every finite time interval. More precisely, for any $T>0$,
		\begin{equation}
			\sup_{t\in[0,T]} ||x||_{L^{\infty}(I)} \leq \max\{|a|,|b|\}.
		\end{equation}
		In particular, the bound is independent of $T$ and $W$, provided that the invariant domain $S$ is fixed.
	\end{remark}
	
	\subsection{Convergence}\label{sec_covergence}
	
	In this section, we prove the convergence of the dynamical system on finite networks  induced by a graphon (i.e., 
	\eqref{eq:pre-graphon-dynamical-system} supplied with \eqref{eq:Aij-from-W}) to that on the graphon (i.e., \eqref{eq4}). We remind that $\{x_1(t), \ldots, x_N(t)\}$ is the solution of \eqref{eq_1} and $x(y,t)$ is the solution of \eqref{eq4}. 
	\begin{lemma}\label{lemma_1}
		Let $\pmb{x}_N$ and $\pmb{x}$ be the solution of the IVP \eqref{eq:pre-graphon-dynamical-system} and \eqref{eq4} with the initial condition $\pmb{x}_{N}(0) = \mathbf{g}_{N} = \left(g_{N1},\ldots,g_{NN} \right)$ and $\pmb{x}(0) = \mathbf{g}$, respectively, where $g_{Ni}= N\int_{I_{N,i}}g(y)\mathrm{d}y$ and $\mathbf{g} \in L^\infty(I)$. We assume that $g(y) \in \Omega$ and $g_{Ni} \in \Omega$. Then,
		\begin{equation}
			\sup_{t\in[0,T]}\left\|\pmb{x}_{N} - \pmb{x}\right\|_{C(L^2(I), [0,T])} \leq \left(\left\|\mathbf{g}_{N} - \mathbf{g}\right\|_{L^2(I)} + C_{1}\left\|W_{N} - W\right\|_{L^2(I^2)}T\right) e^{3LT},
			\label{eq:convergence-lemma-statement}			
		\end{equation}
		where  $C_1$ is a positive constant independent of $N,~y,~y'$ and $t$.
	\end{lemma}
	This lemma is analogous to Theorem 4.1 in Ref.~\cite{medvedev2014nonlinear} including the proof. Therefore, we show the proof in Appendix~\ref{sec:proof-lemma-convergence}.	
	
	\begin{theorem} \label{thm:convergence}
		Let $x_N(y, t) = \sum_{i=1}^N x_i(t) \chi_{I_{Ni}}(y)$ and $x(y,t)$ and be the solution of \eqref{eq:pre-graphon-dynamical-system} and \eqref{eq4}, respectively. Suppose that the initial condition for IVPs \eqref{eq:pre-graphon-dynamical-system} and \eqref{eq4} satisfy 
		\begin{equation}\label{eq_16d}
			\lim_{N\to\infty} \left\| x_{N}(\cdot, 0) - x(\cdot, 0) \right\|_{C(L^2(I),[0,T])} = 0.
		\end{equation}
		Then, it holds true for any fixed $T > 0$ that
		\begin{equation}
			\lim_{N\to\infty} \max_{t \in [0,T]} \left\| x_{N}(\cdot, t) - x(\cdot, t) \right\|_{C(L^2(I),[0,T])} = 0.
		\end{equation}
		
	\end{theorem}
	This theorem is an adaptation of Theorem 5.2 in Ref.~\cite{medvedev2014nonlinear}, and our proof is similar to theirs. We show the proof in Appendix \ref{sub_Sxx}.	
	
	\subsection{Galerkin approximation}\label{sec_galerkin}
	
	In this section, we construct the Galerkin approximation to the solution of the graphon dynamical system, i.e., \eqref{eq4}. We then show that the obtained Galerkin approximation is the solution of the dynamical system on the finite network induced by the graphon, i.e., \eqref{eq_1} (or equivalently \eqref{eq:pre-graphon-dynamical-system}) supplied with \eqref{eq:Aij-from-W}. In \cite{kaliuzhnyi2017semilinear}, it was shown that the Galerkin approximation for a different class of graphon dynamical system converges to the solution of the original graphon dynamical system as $N\to\infty$. We do not attempt to prove similar convergence results for our graphon dynamical system. It is because the equivalence of the Galerkin approximation and \eqref{eq_1}, which we show in this section, combined with the convergence results shown in section~\ref{sec_covergence} guarantees the convergence of the Galerkin solution to that of the graphon dynamical system.
	
	We write the initial condition as
	\begin{equation}\label{eq_34}
		x(y,0) = g(y),
	\end{equation}
	where $\mathbf{g} \in L^{\infty}(I)$.
	
	To simplify the notation, we define an operator $\tilde{K}: L^2(I)\to L^2(I)$ by
	\begin{equation}\label{eq_36}
		\left[\tilde{K}(\pmb{x})\right](y) = \int_{I} W(y,y') G(x(y), x(y'))\mathrm{d}y'.
	\end{equation} 
	Using \eqref{eq_36}, we rewrite \eqref{eq4} as
	\begin{equation}\label{eq_17}
		\dfrac{d\pmb{x}}{dt} = f(\pmb{x}) + \tilde{K}(\pmb{x})
	\end{equation}
	and
	\begin{equation}\label{eq_17a}
		\pmb{x}(0) = \mathbf{g}.
	\end{equation}
	
	\begin{definition}
		Consider ${\pmb{x}}\in H^1(L^2(I); [0,T])$. Then, ${\pmb{x}}$  is a weak solution of \eqref{eq4} with the initial condition given by \eqref{eq_17a} if
		\begin{equation}\label{eq_18b}
			\left({\pmb{x}}'(t) -f({\pmb{x}}(t)) -\tilde{K}[{\pmb{x}}(t)], {\bf v}\right)=0,~~\forall {\bf v}\in L^2(I).
		\end{equation}
		almost everywhere on $[0,T]$ and ${\pmb{x}}(0) = g$.
	\end{definition}
	
	Let $X_N=$span$\{ \chi_{I_{N,i}}:i\in \{1, \ldots,N \} \}$ be a linear subspace of $L^2(I)$; we recall that
	$\chi_{I_{N,i}}$ is the indicator function defined by \eqref{eq_a1}.
	Now we construct the Galerkin approximation to the solution of \eqref{eq_17} and \eqref{eq_17a} in the form
	\begin{equation}\label{eq_18a}
		\pmb{x}_N(t)=\sum\limits_{i=1}^N \tilde{x}_{Ni}(t)\chi_{I_{N,i}}.
	\end{equation}
	
	\begin{theorem}
		Consider the graphon dynamical system given by \eqref{eq4}. The Galerkin approximation to its solution, \eqref{eq_18a},
		is equivalent to the solution of the finite-network dynamical system \eqref{eq_1} induced by the graphon with the initial condition $\pmb{x}_{N}(0) = \mathbf{g} = (g_{N1}, \ldots, g_{NN})$, $g_{Ni} = N\int_{I_{N,i}}g(y)\mathrm{d}y$, $i\in \{1, \ldots, N\}$.
		%
		% \bme{We have already defined the initial condition in Theorem \ref{thm:convergence}, so do we need to say it again here?} \nm{Bit far, so it is fine to keep the initial condition described here like this.}\bme{Okay.}.
		%
	\end{theorem}
	\begin{proof}
		We determine the differentiable coefficients $\tilde{x}_{Ni}(t)$, $i\in \{1, \ldots, N\}$ by solving \eqref{eq_18b} on partition $\mathcal{P}_N$. In other words, we project \eqref{eq_18b} and the initial condition onto $\mathcal{P}_N$, i.e.,
		\begin{equation}\label{eq_19}
			\langle \pmb{x}_N'(t) -f(\pmb{x}_N(t)) -\tilde{K}[\pmb{x}_N(t)], \chi_{I_{N,i}} \rangle=0
		\end{equation}
		for all indicator functions $\chi_{I_{N,i}}$ on partition $\mathcal{P}_{N}$ and
		\begin{equation}\label{eq_19a}
			\pmb{x}_N(0)= P_{\mathcal{P}_N}\mathbf{g}=\sum\limits_{i=1}^N\dfrac{(\mathbf{g},\chi_{I_{N,i}})}{(\chi_{I_{N,i}},\chi_{I_{N,i}})}\chi_{I_{N,i}},
		\end{equation}
		where $P_{\mathcal{P}_N}: L^2(I) \to \mathcal{P}_N$ stands for the orthogonal projection onto $\mathcal{P}_N$. 
		By substituting \eqref{eq_18a} in \eqref{eq_18b} with ${\bf v} = \chi_{I_{N,i}}, i \in \{1, \ldots, N\}$, we obtain
		\begin{equation}
			\langle \pmb{x}_N'(t)(y)-f(\pmb{x}_N(t)(y))-\tilde{K}[\pmb{x}_N(t)](y), \chi_{I_{N,i}}(y)\rangle = 0.
			\label{eq:x'_N-Galerkin-2}	
		\end{equation}
		Because any $y \in [0, 1]$ belongs to $I_{N,i}$ for exactly one $i$, we obtain $\pmb{x}_N(t)(y) = \tilde{x}_{Ni}(t)$ for that $i$ value. 
		Therefore, using $\pmb{x}_N(t)(y) = \tilde{x}_{Ni}(t)$ for $\forall y \in I_{N,i}$, \eqref{eq:Aij-from-W}, and \eqref{eq:x'_N-Galerkin-2}, we obtain
		the IVP
		\begin{equation}\label{eq_20}
			\dot{\tilde{x}}_{Ni}(t)= f(\tilde{x}_{Ni}(t)) + \sum\limits_{i=1}^N (W_N)_{ij} G(\tilde{x}_{Ni}(t),\tilde{x}_{Nj}(t)),
		\end{equation}
		with the initial condition
		\begin{equation}\label{eq_20b}
			\tilde{x}_{Ni}(0)=\dfrac{(\mathbf{g},\chi_{I_{N,i}})}{(\chi_{I_{N,i}},\chi_{I_{N,i}})}=N\int_{I_{N,i}}g(y)\,\mathrm{d}y.
		\end{equation}
		Equation~\eqref{eq_20} is equivalent to \eqref{eq_1} and can be written as
		\begin{equation}\label{eq:galerkin_diff_eq}
			\dfrac{\partial \tilde{x}_N(y,t)}{\partial t}=f(\tilde{x}_N(y,t))+\int_IW_N(y,y')G(\tilde{x}_N(y,t),\tilde{x}_N(y',t))\,\mathrm{d}y'
		\end{equation}
		and
		\begin{equation}\label{eq:galerkin_approx}
			\tilde{x}_N(y, t)=\sum_{i=1}^N\tilde{x}_{Ni}(t)\chi_{I_{N,i}}(y).
		\end{equation}
	\end{proof}
	
	\section{Dimension reduction methods}\label{sec_4}
	
	In this section, we derive a one-dimensional reduced system intended to approximate the graphon dynamical system \eqref{eq4} using two reduction methods: GBB reduction and spectral reduction. The mathematical results in this section establish consistency of the GBB reduction between the finite-network and graphon versions. We also bound the error between the solution of the GBB-reduced system and the corresponding observable of the full graphon dynamical system.
	
	\subsection{GBB reduction}
	
	For the dynamical system on undirected weighted networks given by \eqref{eq_1}, we define an operator $\mathcal{L}_N: \mathbb{R}^N \to \mathbb{R}$ by
	\begin{equation}\label{eq_45}
		\mathcal{L}_N ({\bf u}) = \frac{\sum_{i=1}^{N}\sum_{j=1}^{N}A_{ij}u_{i}}{\sum_{i=1}^{N}\sum_{j=1}^{N}A_{ij}} = \frac{\sum_{i=1}^N s_i u_i} {\sum_{i=1}^N s_i},
	\end{equation}
	where ${\bf u} = (u_1, \ldots, u_N)^\top \in \mathbb{R}^N$ is an $N$-dimensional vector,  and $s_i \equiv \sum_{j=1}^N A_{ij}$ represents the weighted degree (also called the strength) of the $i$th node.
	The GBB reduction \cite{gao2016universal}, which is an one-dimensional approximation to \eqref{eq_1}, is given by
	\begin{equation}\label{eq_46}
		\frac{dx_{N\text{eff}}}{dt}=f(x_{N\text{eff}})+\beta_\text{eff}\,G(x_{N\text{eff}}, x_{N\text{eff}}),
	\end{equation}
	where $\beta_{\text{eff}} = \mathcal{L}_N({\bf s})$, ${\bf s} = (s_1, \ldots, s_N)^{\top} \in \mathbb{R}^N$, and $x_{N\text{eff}} = \mathcal{L}_N(\pmb{x})$.
	
	Now we introduce the GBB reduction for the graphon dynamical system. We have shown that the dynamical system induced by a network which is generated from the graphon, \eqref{eq:pre-graphon-dynamical-system} combined with \eqref{eq:Aij-from-W}, converges to the graphon dynamical system given by \eqref{eq4}. Therefore, we first derive the GBB reduction of the graphon dynamical system.
	By regarding \eqref{eq:Aij-from-W} as a graphon with the identification
	\begin{equation}
		W_{N}(y,y') := (W_N)_{ij} \,\,\,\,\,\, \text{for}~ (y, y') \in I_{N,i}\times I_{N,j}
		\label{eq:W_N-as-graphon}
	\end{equation} and applying \eqref{eq:W_N-as-graphon} to \eqref{eq_45}, we obtain
	\begin{equation}\label{eq_47}
		\mathcal{L}_{N}(\textbf{u}_{N}) = \frac{\int_{0}^{1} \int_{0}^{1}W_{N}(y, y') u_{N}(y)\mathrm{d}y\mathrm{d}y' }{\int_{0}^{1} \int_{0}^{1}W_{N}(y, y')\mathrm{d}y\mathrm{d}y'},
	\end{equation}
	where we interpret the $N$-dimensional vector $\textbf{u}_{N}$ as a piecewise constant function on $[0, 1]$ represented as $\textbf{u}_{N}(y,t) = \sum_{i=1}^{N}u_{i}(t)\chi_{I_{N,i}}(y)$.
	In the limit $N\to\infty$, \eqref{eq_47} becomes
	\begin{equation}\label{eq_48}
		\mathcal{L}(\textbf{u}) = \frac{\int_{0}^{1}\int_{0}^{1} W(y,y') u(y)\mathrm{d}y\mathrm{d}y'}{\int_{0}^{1}\int_{0}^{1} W(y,y')\mathrm{d}y\mathrm{d}y'}
		= \frac{\int_{0}^{1} s(y) u(y) \mathrm{d}y} {\int_{0}^{1} s(y)\mathrm{d}y},
	\end{equation}
	where $\textbf{u} \in L^p(I)$ and $s(y) \in L^1(I)$ is defined in \eqref{eq:def-s(y)}.
	%\begin{equation}
	%s(y) := \int_0^1 W(y, y') dy'.
	%\label{eq:def-s(y)}
	%	\end{equation}
By applying the operator $\mathcal{L}$ to graphon dynamical system, \eqref{eq4}, we obtain
\begin{align}\label{eq_4}
	\dfrac{\partial\mathcal{L}(\pmb{x})}{\partial t}
	%
	% &=\dfrac{\int_0^1k(y)\,\dfrac{\partial}{\partial t}x(y,t)\,dy}{\int_0^1k(y)\,y}
	%
	&=\dfrac{\int_0^1f(x(y,t))\,s(y)\,\mathrm{d}y}{\int_0^1s(y)\,\mathrm{d}y}+\dfrac{\int_0^1\int_0^1W(y,y')\,G(x(y,t),x(y',t))\,s(y)\,\mathrm{d}y'\,\mathrm{d}y}{\int_0^1s(y)\,\mathrm{d}y}.
\end{align}
By adapting the GBB ansatz for finite networks to the case of graphons, we use the following approximations:
\begin{equation}\label{eq_5a}
	\dfrac{\int_0^1f(x(y,t))\,s(y)\,\mathrm{d}y}{\int_0^1s(y)\,\mathrm{d}y}\approx f\left(\dfrac{\int_0^1x(y,t)\,s(y)\,\mathrm{d}y}{\int_0^1s(y)\,\mathrm{d}y}\right)\\=f(\mathcal{L}(\pmb{x}))
\end{equation}
and	
\begin{align}\label{eq_5b}
	\dfrac{\int_0^1\int_0^1W(y,y')\,G(x(y,t),x(y',t))\,s(y)\,\mathrm{d}y'\,\mathrm{d}y}{\int_0^1s(y)\,\mathrm{d}y}&\approx
	\dfrac{\int_0^1\int_0^1W(y,y')\,s(y)\,\mathrm{d}y\,\mathrm{d}y'}{\int_0^1s(y)\,\mathrm{d}y}\,G\left(\dfrac{\int_0^1x(y,t)s(y)\,\mathrm{d}y}{\int_0^1s(y)\,\mathrm{d}y},\dfrac{\int_0^1x(y',t)s(y')\,\mathrm{d}y'}{\int_0^1s(y')\,\mathrm{d}y'}\right)\nonumber\\[14pt]
	&=\mathcal{L}({\bf s})\,G(\mathcal{L}(\pmb{x}),\mathcal{L}(\pmb{x})),
\end{align}
where $\approx$ represents ``approximately equal to''.
By substituting \eqref{eq_5a} and \eqref{eq_5b} in \eqref{eq_4}, we obtain \eqref{eq_46}, with $\beta_{\text{eff}} = \mathcal{L}({\bf s})$ and $x_\text{eff} = \mathcal{L}(\pmb{x})$. 

Since \eqref{eq_5a} and \eqref{eq_5b} are approximation steps underlying the GBB reduction method \cite{gao2016universal}, we quantify their accuracy in the following proposition. Specifically, the proposition bounds the discrepancy introduced by interchanging the GBB reduction operator $\mathcal{L}$ with the self-dynamics \eqref{eq_5a} and the coupling function \eqref{eq_5b}. 
\begin{proposition}\label{prop_appr-bdd}
	Let $x(y,t)$ be a solution of the graphon dynamical system on $[0,T]$. We assume that $x(y,t) \in \Omega$ for almost every $y \in I$ and all $t \in [0, T]$. Define the weighted spread by
	\begin{equation}\label{eq_variance}
		D_s(x (t)) := \frac{\int_{0}^{1} s(y)\left|x(y,t) - \mathcal{L}(\pmb{x}(t))\right| \mathrm{d}y}{\int_{0}^{1}s(y)\mathrm{d}y}
	\end{equation}
	and a constant
	\begin{equation}
		s_{*} := {\text{ess}\sup}_{y\in I}s(y).
	\end{equation}
	Then, we can bound the interchangeability of the GBB reduction operator, $\mathcal{L}$, with the self dynamics and the coupling function as follows:
	\begin{equation}
		\left|\mathcal{L}(f(\pmb{x}(t))) - f(\mathcal{L}(\pmb{x}(t)))\right| \leq L_f \sqrt{D_s(x)}
	\end{equation}
	and 
	\begin{equation}
		\left|\mathcal{L}\left(\int_{0}^{1}W(\cdot, y')G(x(\cdot, t), x(y', t))\mathrm{d}y'\right) - \beta_{\text{eff}}G\left(\mathcal{L}(\pmb{x}(t)), \mathcal{L}(\pmb{x}(t))\right)\right| \leq 2L_G s_{*}\sqrt{D_s(x)}.
	\end{equation}
\end{proposition}
\begin{remark}
	The weighted spread $D_s$ defined in \eqref{eq_variance} represents the weighted absolute deviation of the graphon solution from the GBB observable, where the weight is given by the degree function $s(y)$. 
\end{remark}
\begin{remark}
	This proposition implies that the interchanging the GBB operator with the self-dynamics (i.e., \eqref{eq_5a}) and that with the coupling function (i.e., \eqref{eq_5b}) are accurate when $D_s$ is small, i.e., when $x(y,t)$ is close to $\mathcal{L}(\pmb{x}(t))$ in the $L^2$ sense. However, this closeness is not guaranteed for all graphon dynamical systems used in the numerical simulations in Section \ref{sec_num-res}.
\end{remark}
\begin{proof}
	Using \eqref{eq_48} and \eqref{eq_5a}, we obtain
	\begin{align}
		\left|\mathcal{L}\left(f(\pmb{x}(t))\right) - f\left(\mathcal{L}(\pmb{x}(t))\right)\right| &= \left|\frac{\int_{0}^{1}s(y)f(x(y,t))\mathrm{d}y}{\int_{0}^{1}s(y)\mathrm{d}y} - \frac{\int_{0}^{1}s(y)f(\mathcal{L}(\pmb{x})(t)) \mathrm{d}y}{\int_{0}^{1}s(y)\mathrm{d}y}\right| \nonumber \\
		&= \frac{1}{\int_{0}^{1}s(y)\mathrm{d}y} \left[\int_{0}^{1}s(y)\left|f(x(y,t)) - f(\mathcal{L}(\pmb{x}(t)))\right| \mathrm{d}y\right] 
	\end{align}
	Using Assumption \ref{assumption_1} for the self-dynamics, we obtain
	\begin{align}
		\left|\mathcal{L}\left(f(\pmb{x}(t))\right) - f\left(\mathcal{L}(\pmb{x}(t))\right)\right| &\leq \frac{L_f}{\int_{0}^{1}s(y)\mathrm{d}y} \int_{0}^{1}s(y)\left|x(y,t) - \mathcal{L}(\pmb{x}(t))\right|\mathrm{d}y \nonumber \\
		&\leq \frac{L_f}{\int_{0}^{1}s(y)\mathrm{d}y} \int_{0}^{1}\sqrt{s(y)}\sqrt{s(y)}\left|x(y,t) - \mathcal{L}(\pmb{x}(t))\right|\mathrm{d}y 
		\label{eq:prop1-1}
	\end{align}
	By applying the Cauchy-Schwartz inequality, we obtain
	\begin{align}\label{eq_Lf-bdd}
		(\text{RHS of }\eqref{eq:prop1-1})
		%
		%			\left|\mathcal{L}\left(f(\pmb{x}(t))\right) - f\left(\mathcal{L}(\pmb{x}(t))\right)\right|
		%
		&\leq \frac{L_f \left(\int_{0}^{1} s(y) \mathrm{d}y\right)^{\frac{1}{2}}}{\int_{0}^{1}s(y)\mathrm{d}y}\left(\int_{0}^{1}s(y) \left|x(y,t) - \mathcal{L}(\pmb{x}(t))\right|^{2} \mathrm{d}y\right)^{\frac{1}{2}}\nonumber\\
		&= L_f \left(\frac{\int_{0}^{1}s(y) \left|x(y,t) - \mathcal{L}(\pmb{x}(t))\right|^{2} \mathrm{d}y}{\int_{0}^{1}s(y)\mathrm{d}y}\right)^{\frac{1}{2}}\nonumber\\
		&=  L_f \sqrt{D_s(x)}.
	\end{align}
	Similarly, we consider
	\begin{align}
		&\left|\mathcal{L}\left(\int_{0}^{1}W(\cdot, y')G(x(\cdot, t), x(y', t))\mathrm{d}y'\right) - \beta_{\text{eff}}G\left(\mathcal{L}(\pmb{x}(t)), \mathcal{L}(\pmb{x}(t))\right)\right| \nonumber\\
		=&\left|\frac{1}{\int_{0}^{1}s(y)\mathrm{d}y}\left(\int_{0}^{1}s(y)\int_{0}^{1}W(y, y')G(x(y, t), x(y', t))\mathrm{d}y' \mathrm{d}y -  \beta_{\text{eff}}G\left(\mathcal{L}(\pmb{x}(t)), \mathcal{L}(\pmb{x}(t))\right)\right)\right|.
		\label{eq:prop1-2}
	\end{align}
	By substituting $\beta_{\text{eff}} = \frac{\int_{0}^{1}\int_{0}^{1}s(y)W(y,y')\mathrm{d}y'\mathrm{d}y}{\int_{0}^{1}s(y)\mathrm{d}y}$ in \eqref{eq:prop1-2}, we obtain
	\begin{align}
		& (\text{RHS of }\eqref{eq:prop1-2}) \nonumber\\
		=&\left|\frac{\int_{0}^{1}\int_{0}^{1}s(y)W(y, y')G(x(y, t), x(y', t))\mathrm{d}y' \mathrm{d}y}{\int_{0}^{1}s(y)\mathrm{d}y} -  \frac{\int_{0}^{1}\int_{0}^{1}s(y)W(y,y')G\left(\mathcal{L}(\pmb{x}(t)), \mathcal{L}(\pmb{x}(t))\right)\mathrm{d}y'\mathrm{d}y}{\int_{0}^{1}s(y)\mathrm{d}y}\right|\nonumber \\
		=& \frac{1}{\int_{0}^{1}s(y)\mathrm{d}y} \int_{0}^{1}\int_{0}^{1}s(y) W(y,y') \left|G(x(y, t), x(y', t)) - G\left(\mathcal{L}(\pmb{x}(t)), \mathcal{L}(\pmb{x}(t))\right)\right|\mathrm{d}y'\mathrm{d}y.
		\label{eq:prop1-3}		
	\end{align}
	Using Assumption \ref{assumption_1} for the coupling function, we obtain
	\begin{align}\label{eq_Gbdd}
		& (\text{RHS of }\eqref{eq:prop1-3}) \nonumber\\
		%
		%		&\left|\mathcal{L}\left(\int_{0}^{1}W(y, y')G(x(y, t), x(y', t))dy'\right) - \beta_{\text{eff}}G\left(\mathcal{L}(\pmb{x}(t)), \mathcal{L}(\pmb{x}(t))\right)\right|\nonumber \\
		%
		\leq& \frac{L_G}{\int_{0}^{1}s(y)\mathrm{d}y}\int_{0}^{1}\int_{0}^{1}s(y) W(y,y') \left\|\left(x(y, t), x(y', t)\right) - \left(\mathcal{L}(\pmb{x}(t)), \mathcal{L}(\pmb{x}(t))\right)\right\|_2\mathrm{d}y'\mathrm{d}y\nonumber \\
		=& \frac{L_G}{\int_{0}^{1}s(y)\mathrm{d}y}\left[\int_{0}^{1}\int_{0}^{1}s(y) W(y,y') \left|\left(x(y, t)- \mathcal{L}(\pmb{x}(t))\right)\right|\mathrm{d}y'\mathrm{d}y + \int_{0}^{1}\int_{0}^{1}s(y) W(y,y') \left|\left(x(y', t)- \mathcal{L}(\pmb{x}(t))\right)\right|\mathrm{d}y'\mathrm{d}y\right].
	\end{align}
	Let $I_1 = \frac{1}{\int_{0}^{1}s(y)\mathrm{d}y}\int_{0}^{1}\int_{0}^{1}s(y) W(y,y') \left|\left(x(y, t)- \mathcal{L}(\pmb{x}(t))\right)\right|\mathrm{d}y'\mathrm{d}y$. Then, we obtain
	\begin{align}
		I_1 &= \frac{1}{\int_{0}^{1}s(y)\mathrm{d}y}\int_{0}^{1}s^2(y)\left|\left(x(y, t)- \mathcal{L}(\pmb{x}(t))\right)\right|\mathrm{d}y \nonumber \\
		&\leq \frac{s_{*}}{\int_{0}^{1}s(y)\mathrm{d}y}\int_{0}^{1}s(y)\left|\left(x(y, t)- \mathcal{L}(\pmb{x}(t))\right)\right|\mathrm{d}y.
		\label{eq:prop1-4}
	\end{align}
	By applying the Cauchy-Schwartz inequality, we obtain
	\begin{align}\label{eq_I1-bdd}
		%		I_1
		%
		(\text{RHS of }\eqref{eq:prop1-4}) &\leq \frac{s_{*}\left(\int_{0}^{1}s(y)\right)^{\frac{1}{2}}}{\int_{0}^{1}s(y)\mathrm{d}y}\left(\int_{0}^{1}s(y)\left|\left(x(y, t)- \mathcal{L}(\pmb{x}(t))\right)\right|^2\mathrm{d}y\right)^{\frac{1}{2}} \nonumber \\
		&=\frac{s_{*}}{\left(\int_{0}^{1}s(y)\right)^{\frac{1}{2}}}\left(\int_{0}^{1}s(y)\left|\left(x(y, t)- \mathcal{L}(\pmb{x}(t))\right)\right|^2\mathrm{d}y\right)^{\frac{1}{2}} \nonumber\\
		&=s_{*}\sqrt{D_s(x(t))}.
	\end{align}
	Similarly, let $I_2 = \frac{1}{\int_{0}^{1}s(y)\mathrm{d}y}\int_{0}^{1}\int_{0}^{1}s(y) W(y,y') \left|\left(x(y', t)- \mathcal{L}(\pmb{x}(t))\right)\right|\mathrm{d}y'\mathrm{d}y$. By the same arguments as in the case of $I_1$, we obtain
	\begin{align}\label{eq_I2-bdd}
		I_2 &\leq s_{*}\sqrt{D_s(x(t))}.
	\end{align}
	We substitute \eqref{eq_I1-bdd} and \eqref{eq_I2-bdd} in \eqref{eq_Gbdd} to obtain
	\begin{align}\label{eq_Lg-bdd}
		\left|\mathcal{L}\left(\int_{0}^{1}W(\cdot, y')G(x(\cdot, t), x(y', t))\mathrm{d}y'\right) - \beta_{\text{eff}}G\left(\mathcal{L}(\pmb{x}(t)), \mathcal{L}(\pmb{x}(t))\right)\right| &\leq 2s_{*}L_G\sqrt{D_s(x)}.
	\end{align}
\end{proof}

Now, we establish the accuracy of the GBB observable, $\mathcal{L}(\pmb{x})$, in approximating the solution of the one-dimensional reduced equation of the graphon dynamical system, $x_\text{eff}$, beginning with the following Theorem.

\begin{theorem}\label{thm:accuracy}
	Let $x(y,t)$ be the solution of the IVP for \eqref{eq4}, and
	$\mathcal{L}(\pmb{x}(t))$ be the corresponding GBB observable. Let $x_\text{eff}(t)$ be the solution of the one-dimensional reduced equation for the graphon dynamical system
	(i.e., \eqref{eq_46} but with $x_{N\text{eff}}(t)$ and $\beta_{\text{eff}}$ being replaced by $x_\text{eff}(t)$ and the graphon version of 
	$\beta_{\text{eff}}$, respectively). We assume that $x(y,t) \in \Omega$ for almost every $y \in I$ and for all $t \in [0,T]$, and that $x_\text{eff}(t) \in \Omega$ for all $t \in [0,T]$. Then, for every $t\in [0,T]$, we obtain
	\begin{equation}
		\left|\mathcal{L}(\pmb{x} (t)) - x_{\text{eff}}(t)\right| \leq e^{KT}\left|\mathcal{L}(\pmb{x}(0)) - x_{\text{eff}}(0)\right| + \left(L_f +2s_{*}L_G\right)\int_{0}^{T}e^{K(T- \tau)}\sqrt{D_s(x(\tau))} \mathrm{d}\tau.
	\end{equation}
\end{theorem}
\begin{proof}
	Define the residual by
	\begin{equation}\label{eq_residu}
		\text{Res}_{\text{GBB}}(t) :=  \mathcal{L}\left(\frac{\partial \left(\pmb{x}(\cdot, t)\right)}{\partial t}\right) - \left[f(\mathcal{L}(\pmb{x}(t))) + \beta_{\text{eff}} G(\mathcal{L}(\pmb{x}(t)), \mathcal{L}(\pmb{x}(t)))\right].
	\end{equation}
	Since $\mathcal{L}$ is linear and time-independent, we have 
	\begin{equation}
		\mathcal{L}\left(\frac{\partial (\pmb{x}(\cdot, t))}{\partial t}\right) = \frac{\mathrm{d} \mathcal{L}(\pmb{x}(t))}{\mathrm{d}t}.
	\end{equation}
	Let $e(t) = \mathcal{L}(\pmb{x} (t)) - x_{\text{eff}}(t)$ be the error function. We obtain
	\begin{align}\label{eq:de(t)/dt}
		\frac{\mathrm{d}e(t)}{\mathrm{d}t} &= \frac{\mathrm{d} \mathcal{L}(\pmb{x}(t))}{\mathrm{d}t} - \frac{\mathrm{d} x_{\text{eff}}(t)}{\mathrm{d}t} \\ \nonumber
		&= \mathcal{L}\left(\frac{\partial (\pmb{x}(\cdot, t))}{\partial t}\right) - \left[f\left(x_{\text{eff}}(t)\right) + \beta_{\text{eff}} G\left(x_{\text{eff}}(t), x_{\text{eff}}(t)\right)\right]\\ \nonumber
		&= \left[f(\mathcal{L}(\pmb{x}(t))) + \beta_{\text{eff}} G(\mathcal{L}(\pmb{x}(t)), \mathcal{L}(\pmb{x}(t)))\right] + \text{Res}_{\text{GBB}}(t) - \left[f\left(x_{\text{eff}}(t)\right) + \beta_{\text{eff}} G\left(x_{\text{eff}}(t), x_{\text{eff}}(t)\right)\right] \\ \nonumber
		&= f\left(\mathcal{L}(\pmb{x}(t))\right) - f\left(x_{\text{eff}}(t)\right) + \beta_{\text{eff}} \left[G(\mathcal{L}(\pmb{x}(t)), \mathcal{L}(\pmb{x}(t))) - G\left(x_{\text{eff}}(t), x_{\text{eff}}(t)\right) \right]+ \text{Res}_{\text{GBB}}(t).
	\end{align}
	Using \eqref{eq:de(t)/dt}, we obtain
	\begin{align}
		\left|e'(t)\right| &\leq \left|f\left(\mathcal{L}(\pmb{x}(t))\right) - f\left(x_{\text{eff}}(t)\right)\right| + \left|\beta_{\text{eff}}\right| \cdot \left|G(\mathcal{L}(\pmb{x}(t)), \mathcal{L}(\pmb{x}(t))) - G\left(x_{\text{eff}}(t), x_{\text{eff}}(t)\right)\right|+ \left|\text{Res}_{\text{GBB}}(t)\right|.
		\label{eq:thm7-1}
	\end{align}
	Using Assumption \ref{assumption_1}, we obtain
	\begin{align}
		& (\text{RHS of }\eqref{eq:thm7-1}) \nonumber\\
		%
		%			\left|e'(t)\right|
		%
		\leq& L_f\left|\mathcal{L}(\pmb{x}(t)) - x_{\text{eff}}(t)\right| + \left|\beta_{\text{eff}}\right| L_G \left\|(\mathcal{L}(\pmb{x}(t)), \mathcal{L}(\pmb{x}(t))) - (x_{\text{eff}}(t), x_{\text{eff}}(t))\right\|_2 + \left|\text{Res}_{\text{GBB}}(t)\right|\\ \nonumber
		%	=& L_f\left|\mathcal{L}(\pmb{x}(t)) - x_{\text{eff}}(t)\right| + \left|\beta_{\text{eff}}\right| L_G \left\|(\mathcal{L}(\pmb{x}(t))- x_{\text{eff}}(t), \mathcal{L}(\pmb{x}(t))- x_{\text{eff}}(t)) \right\| + \left|\text{Res}_{\text{GBB}}(t)\right|\\ \nonumber
		=& L_f \left|\mathcal{L}(\pmb{x}(t)) - x_{\text{eff}}(t)\right| + \sqrt{2}\left|\beta_{\text{eff}}\right| L_G \left|\mathcal{L}(\pmb{x}(t))- x_{\text{eff}}(t) \right| + \left|\text{Res}_{\text{GBB}}(t)\right| \\ \nonumber
		=& \left(L_f + \sqrt{2}\left|\beta_{\text{eff}}\right|L_G\right) \left|\mathcal{L}(\pmb{x}(t)) - x_{\text{eff}}(t)\right| + \left|\text{Res}_{\text{GBB}}(t)\right|\\ \nonumber
		=&  \left(L_f + \sqrt{2}\left|\beta_{\text{eff}}\right|L_G\right) \left|e(t)\right| + \left|\text{Res}_{\text{GBB}}(t)\right|.
	\end{align}
	With $K \equiv L_f + \sqrt{2}\left|\beta_{\text{eff}}\right|L_G$, we obtain
	\begin{equation}\label{eq_err-deri-bdd}
		\left|\frac{\mathrm{d}e(t)}{\mathrm{d}t}\right| \leq K \left|e(t)\right| + \left|\text{Res}_{\text{GBB}}(t)\right|.
	\end{equation}
	
	By substituting $\mathcal{L}\left(\frac{\partial \left(\pmb{x}(\cdot, t)\right)}{\partial t}\right) = \mathcal{L}(f(\pmb{x}(t))) + \mathcal{L}\left(\int_{0}^{1}W(\cdot, y')G(x(\cdot, t), x(y', t))\mathrm{d}y'\right)$ in
	\eqref{eq_residu}
	and applying Proposition \ref{prop_appr-bdd}, we obtain
	\begin{align}\label{eq_resi-bdd}
		\left|\text{Res}_{\text{GBB}}(t)\right| &\leq \left|\mathcal{L}(f(\pmb{x}(t))) - f(\mathcal{L}(\pmb{x}(t)))\right| + \left|\mathcal{L}\left(\int_{0}^{1}W(\cdot, y')G(x(\cdot, t), x(y', t))\mathrm{d}y'\right) - \beta_{\text{eff}}G\left(\mathcal{L}(\pmb{x}(t)), \mathcal{L}(\pmb{x}(t))\right)\right|\nonumber\\
		&\leq \left(L_f + 2s_{*}L_G\right)\sqrt{D_s(x)(t)}.
	\end{align}
	By substituting \eqref{eq_resi-bdd} in \eqref{eq_err-deri-bdd}, we obtain
	\begin{align}
		\left|\frac{\mathrm{d}e(t)}{\mathrm{d}t}\right| &\leq K \left|e(t)\right| + \left(L_f + 2s_{*}L_G\right)\sqrt{D_s(x (t))}.
		\label{eq:thm7-2}
	\end{align}
	By applying Gr$\ddot{\text{o}}$nwall's inequality to \eqref{eq:thm7-2}, we obtain
	\begin{equation}
		\left|e (t)\right| \leq e^{Kt}\left|e(0)\right| + \left(L_f + 2s_{*}L_G\right)\int_{0}^{t}e^{K(t-\tau)} \sqrt{D_s(x(\tau))} \mathrm{d}\tau,
	\end{equation}
	namely,
	\begin{equation}
		\left|\mathcal{L}(\pmb{x} (t)) - x_{\text{eff}}(t)\right| \leq e^{Kt}\left|\mathcal{L}(\pmb{x}(0)) - x_{\text{eff}}(0)\right| + \left(L_f + 2s_{*}L_G\right)\int_{0}^{t}e^{K(t-\tau)} \sqrt{D_s(x(\tau))} \mathrm{d}\tau.
		\label{eq:e(t)-bound-1}
	\end{equation}
	Since $K\geq 0$, the RHS of \eqref{eq:e(t)-bound-1} monotonically increases with $t$, so that we obtain
	\begin{equation}
		\left|\mathcal{L}(\pmb{x}(t)) - x_{\text{eff}}(t)\right| \leq e^{KT}\left|\mathcal{L}(\pmb{x}(0)) - x_{\text{eff}}(0)\right| + \left(L_f + 2s_{*}L_G\right)\int_{0}^{T}e^{K(T-\tau)} \sqrt{D_s(x(\tau))} \mathrm{d}\tau.
	\end{equation}
	
\end{proof}

Now, we prove that the dynamical variable for the GBB reduction for finite networks, $x_{N\text{eff}}$, converges to that for the graphon dynamical system, $x_\text{eff}$.
\begin{theorem}\label{thm:dim_red_convergence}
	Let $\pmb{x}_N$ and $\pmb{x}$ be the solution of the IVP \eqref{eq:pre-graphon-dynamical-system} and \eqref{eq4}, respectively. Then, for any fixed $T \geq 0, \, \lim_{N \to \infty} \max_{t \in [0,T]}\left|\mathcal{L}_{N}(\pmb{x}_{N}(t)) - \mathcal{L}(\pmb{x}(t))\right| = 0$.
	%Let $\{ \pmb{x}_{N} \}$ be a sequence of uniformly bounded and almost everywhere continuous functions on $I$ such that $\pmb{x}_{N} \to \pmb{x}$ as $N \to \infty$ pointwise almost everywhere on $I$. Then $\lim_{N \to \infty} \mathcal{L}_{N}(\pmb{x}_{N}) = \mathcal{L}(\pmb{x})$.
\end{theorem}	
\begin{proof}
	Using \eqref{eq_47}, we obtain
	\begin{align}
		\left|\mathcal{L}_{N}(\pmb{x}_{N}(t)) - \mathcal{L}(\pmb{x}(t))\right|
		=& \left|\frac{\int_{0}^{1} \int_{0}^{1}W_{N}(y, y') x_{N}(y, t)\mathrm{d}y\mathrm{d}y' }{\int_{0}^{1} \int_{0}^{1}W_{N}(y, y') \mathrm{d}y\mathrm{d}y'} - \frac{\int_{0}^{1}\int_{0}^{1} W(y,y') x(y,t)\mathrm{d}y\mathrm{d}y'}{\int_{0}^{1}\int_{0}^{1} W(y,y')\mathrm{d}y\mathrm{d}y'}\right|\nonumber\\
		&=\left|\frac{\int_{0}^{1} \int_{0}^{1}\left(W_{N}(y, y') x_{N}(y,t) - W(y,y') x(y,t)\right)\mathrm{d}y\mathrm{d}y'}{\int_{0}^{1}\int_{0}^{1} W(y,y')\mathrm{d}y\mathrm{d}y'}\right|.
		\label{eq:L-convergence-proof-1}
	\end{align}
	We obtain
	\begin{align}\label{eq:Lnorm-convergence-proof-1}
		& (\text{Numerator of \eqref{eq:L-convergence-proof-1}})\nonumber\\
		=& \left|\int_{0}^{1} \int_{0}^{1}\left(W_{N}(y, y') x_{N}(y,t) - W_{N}(y,y') x(y,t) + W_{N}(y, y') x(y,t) - W(y,y') x(y,t)\right)\mathrm{d}y\mathrm{d}y'\right| \nonumber\\
		& \leq \int_{0}^{1}\left[\int_{0}^{1} \left|W_{N}(y,y')\left(x_{N}(y,t) - x(y,t)\right)\right|\mathrm{d}y + \int_{0}^{1}\left|x(y,t)\left(W_{N}(y,y') - W(y,y')\right)\right| \mathrm{d}y\right]\mathrm{d}y'\nonumber \\
		& = \int_{0}^{1}\left\|W_{N}(\cdot,y')\left(x_{N}(\cdot,t) - x(\cdot,t)\right)\right\|_{L^1(I)}\mathrm{d}y' + \int_{0}^{1}\left\|x(\cdot,t)\left(W_{N}(\cdot,y') - W(\cdot,y')\right)\right\|_{L^1(I)}\mathrm{d}y'.
	\end{align}
	By applying the Cauchy-Schwarz inequality to \eqref{eq:Lnorm-convergence-proof-1}, we obtain
	\begin{align}\label{eq:L1_Cauchy}
		& (\text{Right-hand side of \eqref{eq:Lnorm-convergence-proof-1}}) \nonumber\\
		\leq& \int_{0}^{1}\left\|W_{N}(\cdot,y')\right\|_{L^2(I)}\left\|x_{N}(\cdot,t) - x(\cdot,t)\right\|_{L^2(I)}\mathrm{d}y' + \int_{0}^{1}\left\|x(\cdot,t)\right\|_{L^2(I)}\left\|W_{N}(\cdot,y') - W(\cdot,y')\right\|_{L^2(I)}\mathrm{d}y'\nonumber\\
		=& \left\|x_{N}(\cdot,t) - x(\cdot,t)\right\|_{L^2(I)} \int_{0}^{1}\left\|W_{N}(\cdot,y')\right\|_{L^2(I)}\mathrm{d}y' + \left\|x(\cdot,t)\right\|_{L^2(I)} \int_{0}^{1}\left\|W_{N}(\cdot,y') - W(\cdot,y')\right\|_{L^2(I)}\mathrm{d}y'\nonumber\\
		=& \left\|x_{N}(\cdot,t) - x(\cdot,t)\right\|_{L^2(I)} \left\|\left\|W_{N}(\cdot,y')\right\|_{L^2(I)}\right\|_{L^1(I)} + \left\|x(\cdot,t)\right\|_{L^2(I)} \left\|\left\|W_{N}(\cdot,y') - W(\cdot,y')\right\|_{L^2(I)}\right\|_{L^1(I)} \nonumber\\
		\leq& \left\|x_{N}(\cdot,t) - x(\cdot,t)\right\|_{L^2(I)} \left\|\left\|W_{N}(\cdot,y')\right\|_{L^2(I)}\right\|_{L^2(I)} + \left\|x(\cdot,t)\right\|_{L^2(I)} \left\|\left\|W_{N}(\cdot,y') - W(\cdot,y')\right\|_{L^2(I)}\right\|_{L^2(I)} \nonumber\\
		=& \left\|x_{N}(\cdot,t) - x(\cdot,t)\right\|_{L^2(I)}\left\|W_{N}\right\|_{L^2(I^2)} + \left\|x(\cdot,t)\right\|_{L^2(I)}\left\|W_{N} - W\right\|_{L^2(I^2)}.
	\end{align}
	Because $I$ is of finite measure, $\left\|x(\cdot,t)\right\|_{L^2(I)} \leq\left\|x(\cdot,t)\right\|_{L^\infty(I)}$ holds true. By applying this inequality, Remark \ref{corol_1}, Theorem \ref{thm:convergence}, and the condition $\lim_{N \to \infty} \left\|W_{N} - W\right\|_{L^2(I^2)} = 0$ to \eqref{eq:L1_Cauchy}, we obtain
	\begin{align}
		\lim_{N\to\infty} \max_{t \in [0,T]}\left|\int_{0}^{1} \int_{0}^{1}\left(W_{N}(y, y') x_{N}(y,t)\mathrm{d}y\mathrm{d}y' - W(y,y') x(y,t)\right)\mathrm{d}y\mathrm{d}y'\right| = 0.
	\end{align} 
	Therefore, we obtain
	\begin{align}
		\lim_{N\to\infty} \max_{t \in [0,T]} \left|\mathcal{L}_{N}(\pmb{x}_{N}(t)) - \mathcal{L}(\pmb{x}(t))\right| =0.
	\end{align} 
	
\end{proof}	
This result shows that the GBB observable in the finite network converges to the corresponding graphon observable. 
\begin{remark}\label{rem:1}
	Similarly, by applying Assumption \ref{assumption_3} to \eqref{eq:def-s(y)}, we find that the weighted degree of the graphon, $\textbf{s}$, is also a bounded almost everywhere continuous function. Therefore, the proof of Theorem \ref{thm:dim_red_convergence} remains the same when one replaces $\pmb{x}$ by $\textbf{s}$ and use the known result that $\textbf{s}_{N} \to \textbf{s}$ as $N \to \infty$ \cite{bramburger2023pattern}, such that $\lim_{N \to \infty} \mathcal{L}_{N}({\bf s}_{N}) = \mathcal{L}(\bf{s})$ holds true. Therefore, $\beta_{\text{eff}}$ for the discrete network converges to that for the graphon.
\end{remark}
\begin{lemma}\label{lemma_2}
	Let $x_\text{eff}(t)$ be the solution of the one-dimensional reduced equation for the graphon dynamical system
	(i.e., \eqref{eq_46} but with $x_{N\text{eff}}(t)$ and $\beta_{\text{eff}}$ being replaced by $x_\text{eff}(t)$ and the graphon version of 
	$\beta_{\text{eff}}$, respectively). We assume that the initial conditions satisfy
	\begin{equation}\label{eq:1-dimen_ini}
		\lim_{N \to \infty}\left\|x_{N\text{eff}}(0) - x_\text{eff}(0)\right\|_{L^2([0,T])} = 0.
	\end{equation}
	Then, it holds true that for any fixed $T>0$
	\begin{equation}\label{eq:1-dimen_cgs}
		\lim_{N \to \infty}\max_{t \in [0,T]}\left\|x_{N\text{eff}}(t) - x_\text{eff}(t) \right\|_{L^2([0,T])} = 0.
	\end{equation}
\end{lemma}
We prove this Lemma in Appendix~\ref{sup_sec: cgs 1-dimen system}.	

\subsection{Spectral reduction}	
The spectral reduction is a generalization of the GBB reduction and formulated as follows
\cite{laurence2019spectral, thibeault2020threefold, masuda2022dimension, vegue2023dimension}.
Define operator $\tilde{\mathcal{L}}_N$ by
\begin{equation}
	\tilde{\mathcal{L}}_N(\pmb{x}_N) = \sum_{i=1}^{N}a_{i}x_{i}  = \textbf{a}^\top \pmb{x}_N,
	\label{eq:def-R}
\end{equation}
where $\textbf{a} = (a_1, \ldots, a_N)$ represents the right leading eigenvector of $A^{\top}$, with the associated leading eigenvalue $\alpha_N$, normalized as $\sum_{i=1}^{N}a_{i} = 1$. Because $A = (A_{ij})$ is the adjacency matrix of the network, which we assume to be connected, $A$ is a primitive matrix. Therefore, the Perron-Frobenius theorem implies that all entries of vector $\textbf{a}$ have the same sign, which we set to be positive.	
The one-dimensional spectral reduction of \eqref{eq_1} is given by
\begin{equation}
	\frac{dR_N}{dt} = f(R_N) + \alpha_N G(R_N, R_N),
	\label{eq:SM-original}
\end{equation}	
where $R_N = \tilde{\mathcal{L}}_N(\pmb{x})$ is an observable and analogous to $x_{N\text{eff}}$ in the GBB reduction. The original spectral reduction uses
$G(\tilde{\beta} R_N, R_N)$ with a network-dependent constant $\tilde{\beta}$ instead of $ G(R_N, R_N)$ on the right-hand side of \eqref{eq:SM-original} \cite{laurence2019spectral, thibeault2020threefold}. However, we use \eqref{eq:SM-original} because
this form is directly derived from the standard Taylor expansion and attains a good accuracy \cite{masuda2022dimension}.

For the solution of the graphon dynamical system, $x(y, t)$, we consider the continuum limit of \eqref{eq:def-R} to define
\begin{equation}
	R: = \tilde{\mathcal{L}}(\pmb{x}) = \int_{0}^{1}a(y)x(y,t)\,\mathrm{d}y,
\end{equation}
with $\int_0^1a(y)\,\mathrm{d}y=1$. Specifically, we consider the case in which $a(y)$ is the leading eigenfunction of the graphon $W$, i.e., the eigenfunction associated with the leading eigenvalue of $W$, denoted by $\alpha$.
By following the derivation of the spectral reduction for finite networks \cite{laurence2019spectral, thibeault2020threefold},
we obtain the same one-dimensional form, \eqref{eq:SM-original}, in which we replace $R_N$ and $\alpha_N$ by $R$ and $\alpha$, respectively. We show the derivation in Appendix~\ref{sec:derivation-SM-graphon}.

\section{Numerical results}\label{sec_num-res}

In this section, we numerically analyze the accuracy of the dimension reduction methods on graphon dynamical systems. We use six dynamical systems and six types of graphons for each dynamical system.
A graphon dynamical system, \eqref{eq4}, is an integro-differential equation. To find their solutions numerically, we use the methods described in Ref.~\cite{day1967note}. We first compute the integral involving the graphon using the Simpson's rule. Once we compute the integral term, the integro-differential equation reduces to a normal ODE, which we solve by the Runge-Kutta method of degree 4.

\subsection{Dynamical system models}

Here we describe the six dynamical system models used, focusing on their graphon form.

The susceptible-infectious-susceptible (SIS) model represents spreading of endemic diseases in a population.
Its deterministic, graphon form is given by \cite{vizuete2020graphon, delmas2022infinite}
\begin{equation}\label{eq_21}
	\dfrac{\partial x(y,t)}{\partial t} = -\mu x(y,t) + D \int_0^1W(y,y')(1-x(y,t))x(y',t)\mathrm{d}y',
\end{equation}
where $x(y,t)$ is the probability that node $y$ is infected at time $t$; $\mu$ represents the recovery rate; $D$ represents the infection rate. 
The first term on the right-hand side of \eqref{eq_21} represents the recovery. The second term represents the rate at which node $y$ is infected by node $y'$. Without loss of generality, we set $\mu = 1$; multiplying $\mu$ and $D$ by a common positive constant does not change the equilibria of \eqref{eq_21}, which GBB and spectral reductions approximate. We use the initial condition $\pmb{x}(0) = 1$. Equation~\eqref{eq_21} is a case of \eqref{eq4} with $f(x) = - \mu x$ and $G(x, x') =  D (1-x(y))x'$. The GBB reduction of \eqref{eq_21} is given by
\begin{equation}\label{eq_21a}
	\dfrac{\mathrm{d} x_{\text{eff}}}{\mathrm{d}t} = - \mu x_{\text{eff}} + D \beta_{\text{eff}}(1-x_\text{eff})x_{\text{eff}}.
\end{equation}
The spectral reduction is given by \eqref{eq_21a} with $x_\text{eff}$ and $\beta_{\text{eff}}$ being replaced by $R$ and $\alpha$, respectively. This correspondence also holds true for the other dynamical systems. Therefore, for the following five dynamical systems, we only show the GBB reduction.

The coupled double-well system represents interacting bistable dynamics on nodes \cite{brummitt2015coupled, kronke2020dynamics, wunderling2020motifs}. Its graphon form is given by
\begin{equation}\label{eq_22}
	\dfrac{\partial x(y,t)}{\partial t}  = -(x(y,t) - r_1)(x(y,t) - r_2)(x(y,t) - r_3) + D\int_0^1W(y,y')x(y',t) \mathrm{d}y',
\end{equation}
where $r_1, r_1, r_3$ are constants such that $r_1<r_2<r_3$ and $D$ is the coupling strength. 
We set $r_1 = 1, r_2 = 2$, and $r_3 = 5$. When the nodes are isolated, this system has two stable equilibria, one at $x^{*}< r_1$, referred to as the lower state, and $x^{*}>r_3$, referred to as the upper state; they can be switched between via saddle-node bifurcations. We initialize the dynamics at two initial conditions. The first initial condition is near the lower state, i.e., $\pmb{x}(0) = 0$. The second initial condition is near the upper state, i.e., $\pmb{x}(0) = 5$. If the nodes are initially in their lower states and $D$ gradually increases, the states of various nodes are expected to jump to their upper states, presumably via saddle-point bifurcations~\cite{kundu2022mean}. Different nodes may jump at different values of $D$ depending on node index $y$. The GBB reduction of the coupled double-well system is given by
\begin{equation}\label{eq_23}
	\dfrac{\mathrm{d} x_\text{eff}}{\mathrm{d}t} = -(x_\text{eff} - r_1)(x_\text{eff} - r_2)(x_\text{eff} - r_3) + D\beta_\text{eff}x_\text{eff}.
\end{equation}

The graphon version of a gene-regulatory dynamics model, which has been employed in particular in studies of GBB and spectral reductions~\cite{gao2016universal}, is given by
\begin{equation}\label{eq_25}
	\dfrac{\partial x(y,t)}{\partial t} = -B x^f(y,t) + D \int_{0}^{1}W(y,y')\frac{x^h(y,t)}{1+x^h(y',t)}\mathrm{d}y'.
\end{equation}
We set $B = 1, f = 1$, and $h = 2$ by following Ref.~\cite{gao2016universal}. In \eqref{eq_25}, $x(y,t)$ represents the expression level of gene $y$. The first term on the right-hand side represents the degradation of gene expression. The second term represents the activation of gene $y$ by gene $y'$. We initialize the dynamics at the lower state, i.e., $\pmb{x}(0) = 0$, or the upper state, i.e., $\pmb{x}(0) = 10$. The GBB reduction of \eqref{eq_25} is given by
\begin{equation}\label{eq_26}
	\dfrac{\mathrm{d} x_\text{eff}}{\mathrm{d}t} = -Bx_\text{eff}^f + D\beta_\text{eff}\frac{x_\text{eff}^h}{1+x_\text{eff}^h}.
\end{equation}

The generalized Lotka-Volterra (GLV) model represents the dynamics of multi-species interaction \cite{tu2017collapse, gonze2018microbial}. Its graphon form is given by
\begin{equation}\label{eq_51}
	\dfrac{\partial x(y,t)}{\partial t} = \tilde{\alpha}x(y,t) -c x(y,t)^2 + D \int_{0}^{1}W(y,y')x(y,t)x(y',t)\mathrm{d}y',
\end{equation}
where $x(y,t)$ represents the abundance of species $y$; $\tilde{\alpha}$ is the intrinsic growth rate of the species; $c$ is a constant; $D$ is the coupling strength. By following the same algebra as in Ref.~\cite{kundu2022accuracy}, we obtain $x(y,t) = -\tilde{\alpha}(DF - c)^{-1}\mathds{1}$,  where $F: C(L^2(I), [0,T]) \to C(L^2(I), [0,T])$ is defined by $F(x(y,t)) = \int_{0}^{1}W(y,y')x(y',t)\mathrm{d}y'$, and $\mathds{1}$ is the constant function defined by $\mathds{1}(y) = 1$, $\forall y\in [0, 1]$. To obtain a finite solution, we need $c > D\alpha$, and we consider the range of $D\alpha \in [0,1]$ in the numerical simulation of this model. Therefore, we set $c = 1.1$. We set $\tilde{\alpha} = 0.5$\cite{tu2017collapse}. %For simplicity, we set $c = \alpha + 0.1$, where we remind that $\alpha$ is the largest eigenvalue of the graphon. Furthermore, all the eigenvalues of any graphon $W$ lies in $[-1,1]$ \cite{lovasz2012large}. We set $\tilde{\alpha} = 0.5$ and $c = 1.1$  
We initialize the dynamics by $\pmb{x}(0) = 0$. The GBB reduction of the GLV model is given by
\begin{equation}\label{eq_24a}
	\dfrac{\mathrm{d} x_\text{eff}}{\mathrm{d}t} = \alpha x_\text{eff} + (D\beta_\text{eff}-c)x_\text{eff}^2.
\end{equation}

We use a dynamical system of the cooperative mutualistic interactions between species in an ecosystem \cite{gao2016universal}.  The graphon version of this dynamics is given by
\begin{align}
	\dfrac{\partial x(y,t)}{\partial t} &= B + x(y,t)\left[1-\frac{x(y,t)}{K}\right] \left[\frac{x(y,t)}{\tilde{C}} -1\right] \nonumber\\  
	&+ D \int_{0}^{1}W(y,y')\frac{x(y,t)x(y',t)}{\tilde{D} + E x(y,t)+ H x(y',t)}\mathrm{d}y',
	\label{eq_27}
\end{align}
where $x(y,t)$ represents the abundance of species $y$. Term $B$ represents the constant migration rate of species $y$ from outside the ecosystem; the second term on the right-hand side of \eqref{eq_27} represents the logistic growth with the carrying capacity $K$ and  the Allee constant $\tilde{C}$; the third term represents mutualistic interaction, i.e., the influence of $x(y',t)$ on $x(y,t)$; $\tilde{D}$, $E$, and $H$ are additional constants. We set $B = 0.1$, $\tilde{C} = 1$, $\tilde{D} = 5$, $E = 0.9$, $H = 0.1$, and $K=5$ by following Ref.~\cite{gao2016universal}. We initialize the dynamics at the lower state, i.e., $\pmb{x}(0) = 0$, or the upper state, i.e., $\pmb{x}(0) = 10$. The GBB reduction of \eqref{eq_27} is given by
\begin{equation}\label{eq_28}
	\dfrac{\mathrm{d} x_{\text{eff}}}{\mathrm{d}t} = B + x_{\text{eff}}\left(1-\frac{x_{\text{eff}}}{K}\right)\left(\frac{x_{\text{eff}}}{\tilde{C}} -1\right) 
	+ \frac{D\beta_{\text{eff}} x_{\text{eff}}^2}{\tilde{D} + E x_\text{eff}+ H x_\text{eff}}.
\end{equation}

The Wilson-Cowan model describes the firing rates of populations of synaptically coupled excitatory and inhibitory neurons~\cite{wilson1973mathematical}. The graphon variant of the 
Wilson-Cowan model~\cite{bramburger2024persistence} is given by
\begin{align}\label{eq_29}
	\dfrac{\partial x(y,t)}{\partial t} =-x(y,t) + D \int_{0}^{1}\frac{W(y,y')}{1 + \exp[\mu - \delta(x(y',t))]}\mathrm{d}y',
\end{align}
where $x(y, t)$ represents the excitation level of neuron $y$, which decays due to the linear self-interaction term and is sustained by couplings to other nodes; $\mu$ and $\delta$ are parameters that define the activation threshold.
We set $\mu = 3$ and $\delta = 1$ \cite{laurence2019spectral}. We initialize the dynamics at the lower state, i.e., $\pmb{x}(0) = 0$, or the upper state, i.e., $\pmb{x}(0) = 8$. 
The GBB reduction of \eqref{eq_29} is given by
\begin{equation}\label{eq_30}
	\dfrac{\mathrm{d} x_{\text{eff}}}{\mathrm{d}t} = -x_\text{eff} +  \frac{D\beta_\text{eff}}{1 + \exp[\mu - \delta(x_\text{eff})]}.
\end{equation}
\begin{remark}
	For each of the six dynamical systems considered here, and for all the parameter ranges used for our numerical simulation, the following holds true: we can find a bounded interval $\Omega \subset \mathbb{R}$ such that the induced set $S \subset L^q(I)$ which takes values in $\Omega$ almost everywhere is positively invariant for the corresponding graphon dynamical system.
\end{remark}

\subsection{Graphons}

We use the following six graphons in our numerical simulations.
They are used in the previous literature on graphon dynamical systems or capture some features of real-world networks. 

\begin{figure}[h]\label{fig_1}
	\centering\includegraphics[scale=0.5]{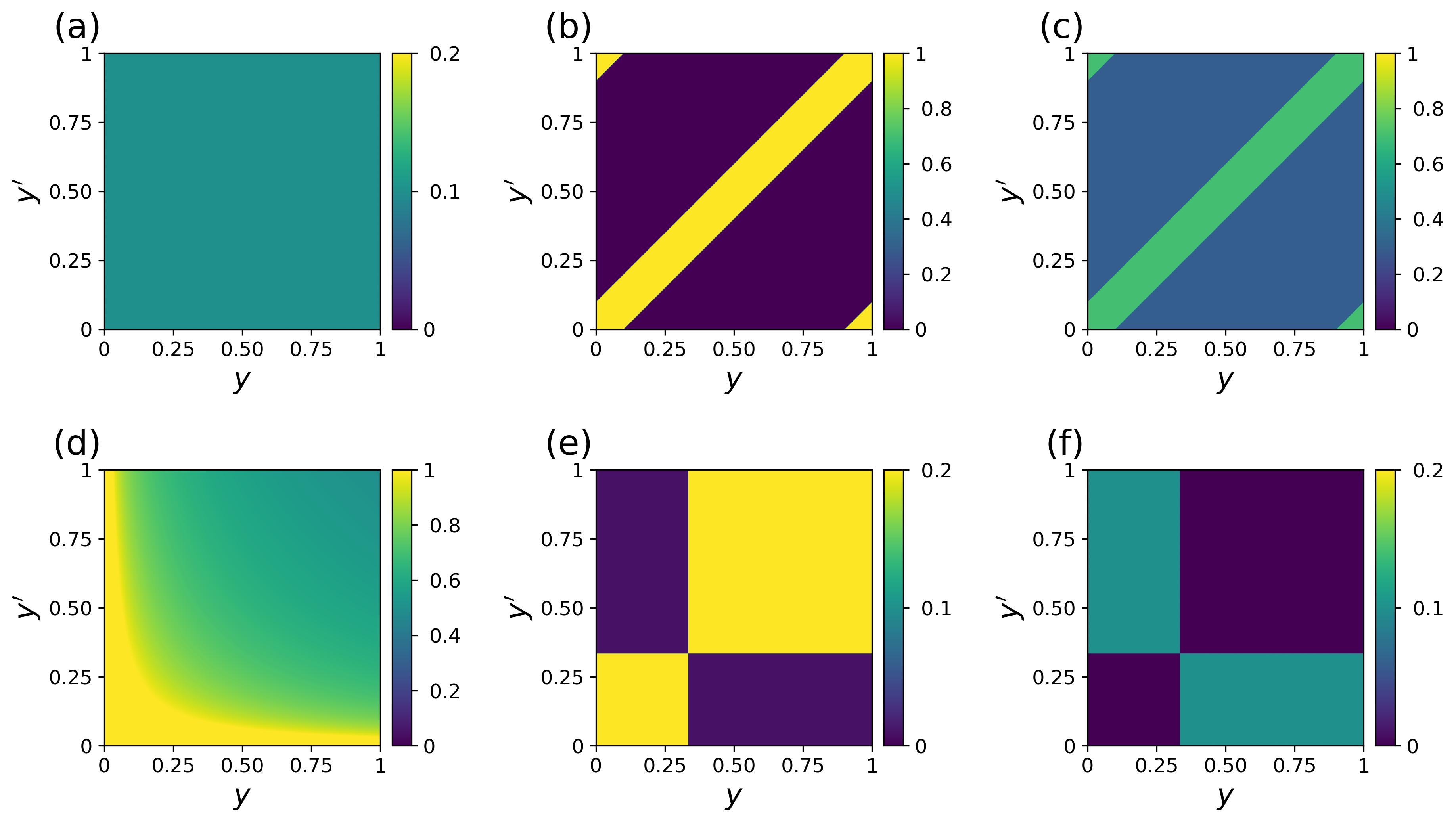}
	\caption{Visualization of the graphons used in the present study. (a) ER graphon with $p = 0.1$. (b) Ring graphon with $q=1/3$. (c) Small-world graphon with $p = 0.1$ and $q = 1/3$. (d) Power-law graphon with $C= 0.5$ and $\nu = -0.2$. (e) Modular graphon with $\gamma = 1/3$, $p_{\rm in} = 0.2$, and $p_{\rm out} = 0.01$. (f) Bipartite graphon with $\gamma = 1/3$ and $p = 0.1$.}
	\label{fig_graphonkernel}
\end{figure}

The Erd\H{o}s-R\'enyi (ER) grahpon (see Fig.~\ref{fig_graphonkernel}(a) for a visualization) is given by 
\begin{equation}
	W(y,y') = p,
\end{equation}
where $0\le p\le1$ is the probability of an edge between any pair of nodes~\cite{chiba2016mean, bramburger2024persistence} (also called the $1$-block graphon \cite{borgs2018revealing}). Its step graphon is the ER random graph. We set $p=0.1$.

The ring graphon (see Fig.~\ref{fig_graphonkernel}(b)) models a ring network in which edges only connect node pairs whose distance along the circumference of the ring is smaller than a threshold, $q$ \cite{medvedev2014nonlinear, bramburger2023pattern, bramburger2024persistence}. The ring graphon is given by
\begin{equation}\label{eq_1**}
	W(y,y') = \begin{cases*}
		1 & if  $\min\{|y-y'|, 1-|y-y'|\}\leq q$, \\
		0 & otherwise,
	\end{cases*} 
\end{equation}
where $0<q<1$. We set $q = 1/3$.

The small-world graphon \cite{medvedev2014nonlinear, kuehn2019power} (see Fig.~\ref{fig_graphonkernel}(c)) is a graphon approximation to the Watts-Strogatz small-world network model~\cite{watts1998collective}. In this model, nodes are placed on the unit-circumference circle and locally connected with a large probability similar to the case of the ring graphon. In addition, each node is connected to any other node with a small probability. The small-world graphon is given by
\begin{equation}
	W(y,y')=(1-p)W_q(y,y')+p\left[1-W_q(y,y')\right],
\end{equation}
where $W_q$ represents the ring graphon given by \eqref{eq_1**}, and $0\le p\le0.5$.
We obtain a ring graphon if $p=0$ and an ER graphon if $p=0.5$. We set $p=0.1$ and $q = 1/3.$	

The power-law graphon (see Fig.~\ref{fig_graphonkernel}(d)) is defined by $W(y,y')= \min\{C(yy')^{\nu}, 1\}$, where $0 < C<1$ and $-0.5 < \nu < 0$. The present power-law graphon is similar to the model proposed in Refs.~\cite{van2018sparse, medvedev2020kuramoto}. The weighted degree of the node for any $y$ is given by \eqref{eq:def-s(y)}. Because $W(y,y') = 1$ if $y'< C^{-\frac{1}{\nu}} y^{-1}$ and $W(y,y') = C(yy')^{-\nu}$ if $y'\ge C^{-\frac{1}{\nu}} y^{-1}$, we obtain
\begin{align}\label{eq:deg_distribution}
	s(y) &= \int_{0}^{C^{-\frac{1}{\nu}} y^{-1}} \mathrm{d}y' + \int_{C^{-\frac{1}{\nu}} y^{-1}}^{1}C\left(yy'\right)^{-\nu}\mathrm{d}y'\nonumber\\
	&= \frac{C^{-\frac{1}{\nu}}y^{-1} \nu + Cy^{\nu}}{1+\nu}.
\end{align} 
Because we have assumed that $-0.5 < \nu < 0$,	we obtain
\begin{align}
	s(y) \propto \frac{1}{y}
\end{align}
as $y\to 0$, where $\propto$ represents ``in proportion to''.	
%	Therefore,
%	\begin{align}\label{eq:degree_dist1}
	%		y \propto \frac{1}{s(y)}
	%	\end{align}
Because $y \in [0,1]$ is uniformly distributed, we obtain $p(y) = 1, \forall y \in [0,1]$. By substituting this in
\begin{equation}
	p(s) =  p(y)\frac{\mathrm{d}y}{\mathrm{d}s}
\end{equation}
and using \eqref{eq:deg_distribution}, we obtain
%
%	\begin{align}
	%		\frac{dy}{ds} &= \frac{\alpha +1}{\alpha\left(-C^{\frac{-1}{\alpha}}y^{-2} + Cy^{\alpha -1}\right)}
	%	\end{align}
%
%    Applying \eqref{eq:degree_dist1}, we obtain
%	\begin{align}
	%		\frac{dy}{ds} &= \frac{\alpha +1}{\alpha\left(-C^{\frac{-1}{\alpha}}s^{2}(y) + Cs^{1+\alpha}(y)\right)}\nonumber\\
	%		&= \frac{\alpha +1}{\alpha s^{2}(y)\left(-C^{\frac{-1}{\alpha}}+ Cs^{1-\alpha}(y)\right)}
	%	\end{align}
%	Thus,
\begin{equation}
	p(s) = \frac{\nu +1}{\nu s^{3+\nu} \left[-C^{-\frac{1}{\nu}}s^{-(1 +\nu)}+ C \right]} \propto s^{-(3+\nu)}
\end{equation}
as $s\to\infty$, equivalently $y\to 0$, implying a power-law degree distribution. We set $C = 0.5$ and $\nu = -0.2$.

The modular graphon \cite{klimm2022modularity} models a network with community structure. In networks with community structure,
nodes in the same community are more likely to be adjacent to each other than nodes in different communities are. We consider a modular graphon with two modules (i.e., communities) by dividing $[0,1]$ into $[0,\gamma)$ and $[\gamma,1]$, where $\gamma \in (0, 1)$. We denote the probability of an edge within the same community and between different communities by $p_{\text{in}}\in[0,1]$ and $p_\text{out} \in [0,1]$, respectively,  with $p_\text{in} \gg p_\text{out}$. The modular graphon (see Fig.~\ref{fig_graphonkernel}(e)) is defined by
\begin{equation}
	W(y,y') = \begin{cases}
		p_{\text{in}} & \text{if } y,y' \in [0,\gamma) \text{ or } y,y' \in [\gamma,1],\\
		p_{\text{out}} & \text{otherwise}.
	\end{cases} 
\end{equation}
We set $p_{\text{in}}=0.2$, $p_{\text{out}}=0.01$, and $\gamma= 1/3$.

By definition, a bipartite network consists of two subsets of nodes, and the edges exist only between two nodes in the different subsets.
The bipartite graphon \cite{bramburger2024persistence}
consists of the interval $[0,1]$ partitioned into two segments, $[0,\gamma)$ and $[\gamma,1]$, corresponding to the two subsets of nodes, where $\gamma\in (0,1)$. 
The bipartite graphon (see Fig.~\ref{fig_graphonkernel}(f)) is defined by
\begin{equation}
	W(y,y') = \begin{cases}
		p & \text{if } y < \gamma \le y' \text{ or } y' < \gamma \le y,\\
		0 & \text{otherwise}.
	\end{cases}
\end{equation}
We set $p=0.1$ and $\gamma = 1/3$.

\subsection{Error measurement}

We measure the relative error, denoted by $\text{RE}_{\text{GBB}}$, between the two dimension reduction methods and the full graphon dynamical system as follows. We compute the value of $D\beta_{\text{eff}}$, $\mathcal{L}(x)$, and $x_\text{eff}$  by varying the coupling strength $D$ of each dynamics model. We compute the relative error for the GBB reduction at each $D$ value by
\begin{equation}\label{eq_31}
	\text{RE}_{\text{GBB}} = \left|\frac{\mathcal{L}(\pmb{x}) - x_\text{eff}}{\mathcal{L}(\pmb{x})}\right|.
\end{equation}
We only use the equilibrium values of the dynamics to evaluate the relative error.
The relative error for the spectral method is calculated in the same manner, with the replacement of $\mathcal{L}$ by $\tilde{\mathcal{L}}$ and $x_{\text{eff}}$ by $R$.
To obtain the relative error across a range of $D$ values, we integrate \eqref{eq_31} over $D\beta_{\text{eff}} \in [0, 20]$ and 
$D\alpha \in [0, 20]$ for the GBB and spectral reductions, respectively.

\clearpage
\subsection{Numerical results}

In the panels on the left column in Fig.~\ref{fig_SIS_Model}, we compare our one-dimensional observable $\mathcal{L}(x)$ 		computed directly from the SIS model, \eqref{eq_21}, and the solution $x_\text{eff}$ of the GBB reduction, \eqref{eq_21a}, as a function of $D\beta_\text{eff}$. The panels on the right column of the same figure show
the corresponding results for the spectral reduction. Each figure panel represents a pair of reduction method (i.e., GBB or spectral) and one of the six graphons. The figure indicates that both GBB and spectral reductions provide an accurate approximation across all graphons.

\begin{figure}[h!]
	\centering
	\includegraphics[scale=0.48]{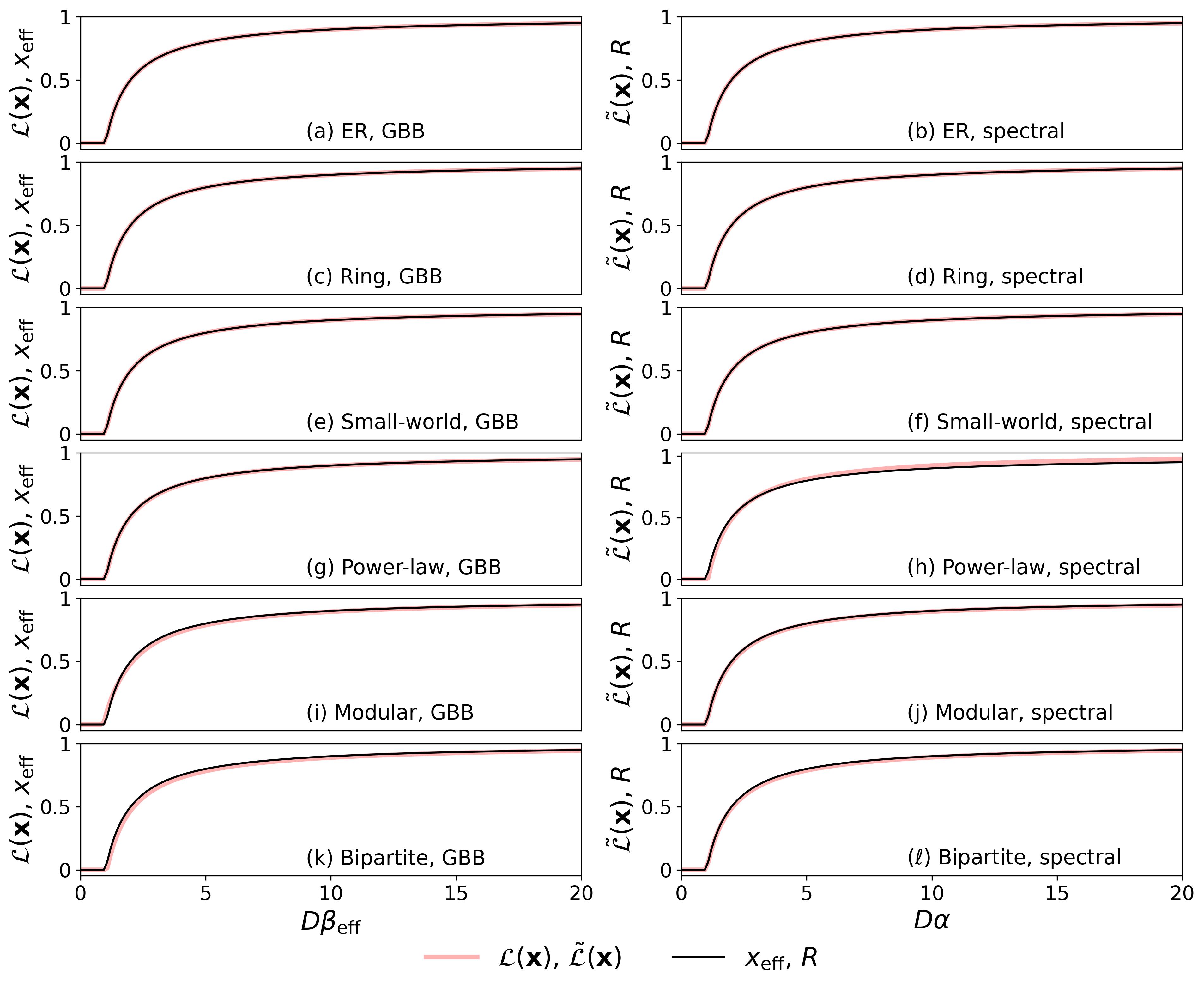}
	\caption{Accuracy of the two dimension reduction methods for the SIS dynamics.
		(a) ER graphon, GBB reduction. (b) ER graphon, spectral reduction. (c) Ring, GBB. (d) Ring, spectral. (e) Small-world, GBB. (f) Small-world, spectral. (g) Power-law, GBB. (h) Power-law, spectral. (i) Modular, GBB. (j) Modular, spectral. (k) Bipartite, GBB. (l) Bipartite, spectral. The thick semi-transparent red lines represent the numerically obtained one-dimensional observable obtained from simulations of the full graphon dynamical system. The thin black lines represent the equilibria of the GBB or spectral reduction.}
	\label{fig_SIS_Model}
\end{figure}
\clearpage

In Fig.~\ref{fig_DW_Model}, we show similar comparison results for the coupled double-well dynamics. We find that
both GBB and spectral reductions are satisfactory for the ER, ring, and small-world graphons. For the other three graphons, both GBB and spectral reductions are inaccurate at locating the bifurcation points to different extents.

\begin{figure}[h!]
	\centering
	\includegraphics[scale=0.48]{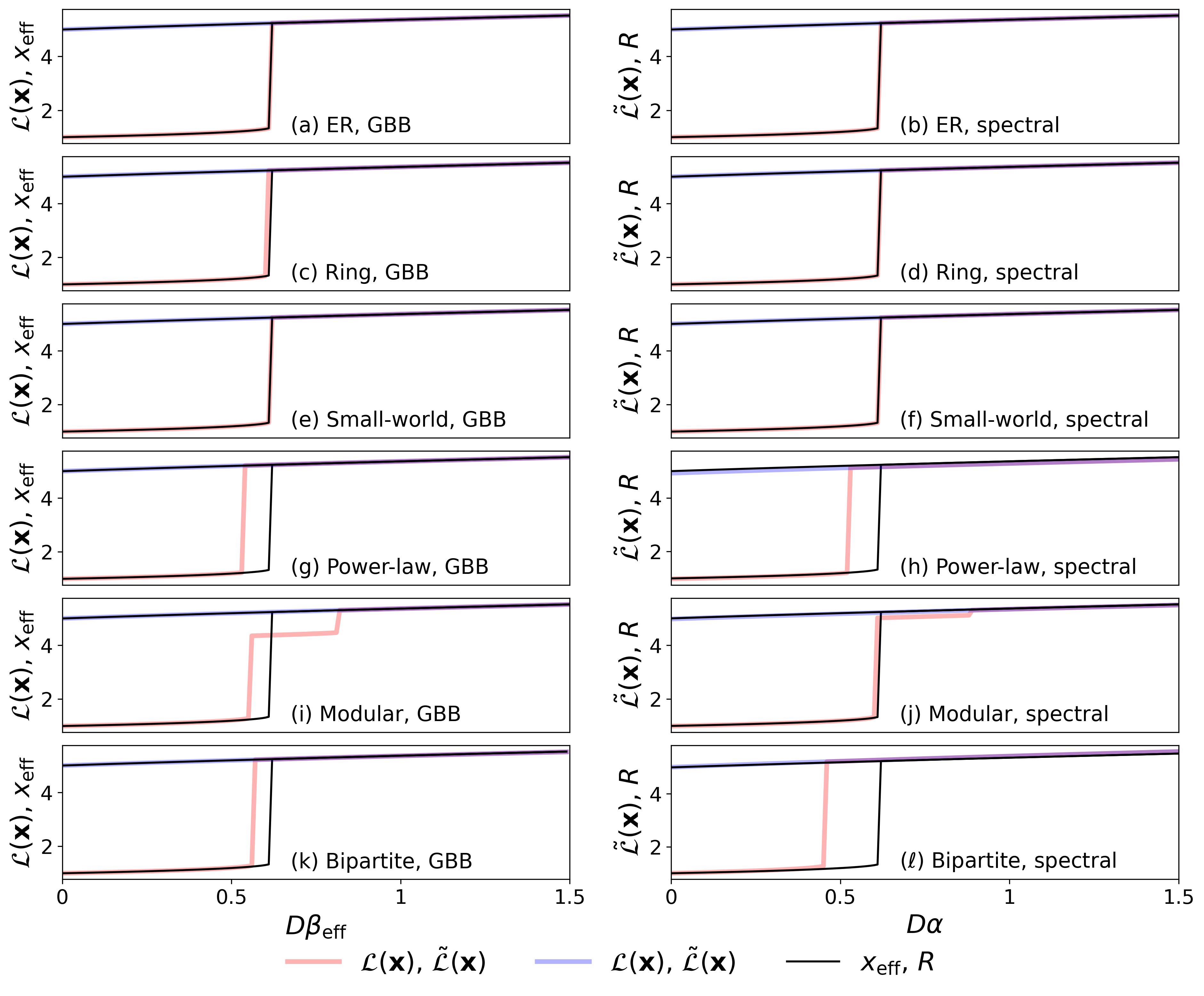}
	\caption{Accuracy of the two dimension reduction models for the coupled double-well dynamics. (a) ER, GBB. (b) ER, spectral. (c) Ring, GBB (d) Ring, spectral. (e) Small-world, GBB. (f) Small-world, spectral. (g) Power-law, GBB. (h) Power-law, spectral. (i) Modular, GBB. (j) Modular, spectral. (k) Bipartite, GBB. (l) Bipartite, spectral. The thick semi-transparent red and blue lines represent the numerically obtained one-dimensional solutions when the initial states are lower and upper, respectively. The thin black lines represent the equilibria of the GBB or spectral reductions with the lower and upper initial conditions altogether.}
	\label{fig_DW_Model}
\end{figure}

Figure \ref{fig_GR_Model} shows the results for the gene-regulatory dynamics. We find that the GBB reduction provides an accurate approximation across all the six graphons. In contrast, the spectral reduction is sufficiently accurate only for the ER, ring, power-law, and modular graphons.

\clearpage

\begin{figure}[h!]
	%	\begin{center}
		\centering
		\includegraphics[scale=0.48]{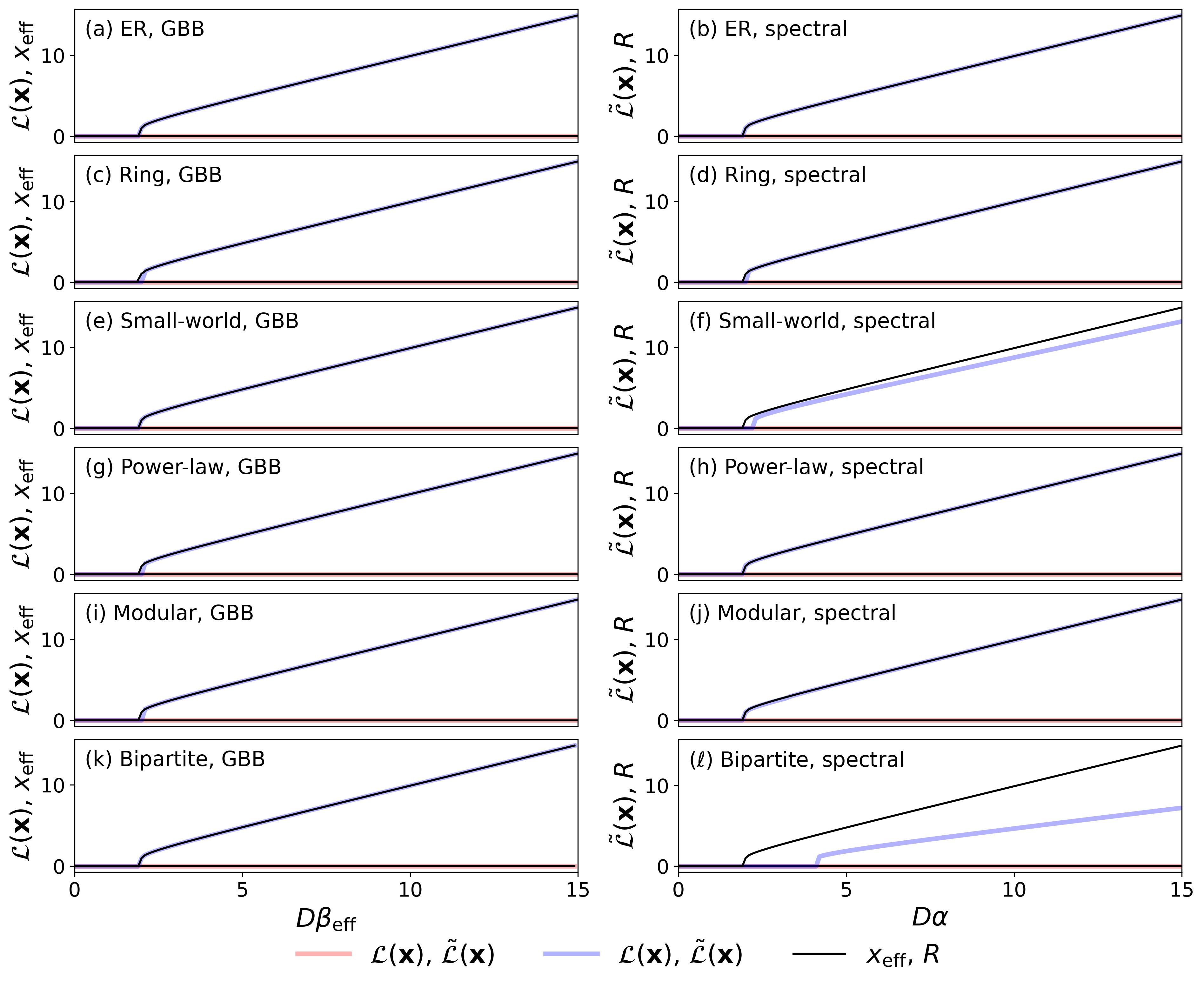}
		\caption{Accuracy of the two dimension reduction models for the gene-regulatory dynamics. See the caption of Fig.~\ref{fig_DW_Model} for the legends.
		}
		%\end{center}
		\label{fig_GR_Model}
	\end{figure}
	
	Figure~\ref{fig_GLV_Model} shows the results for the GLV model. We observe that the spectral reduction is sufficiently accurate except in the case of the power-law graphon. In contrast, the GBB reduction produces notable error at large coupling strengths for the modular and bipartite graphons.
	
	\clearpage
	\begin{figure}[h!]
		\centering
		\includegraphics[scale=0.48]{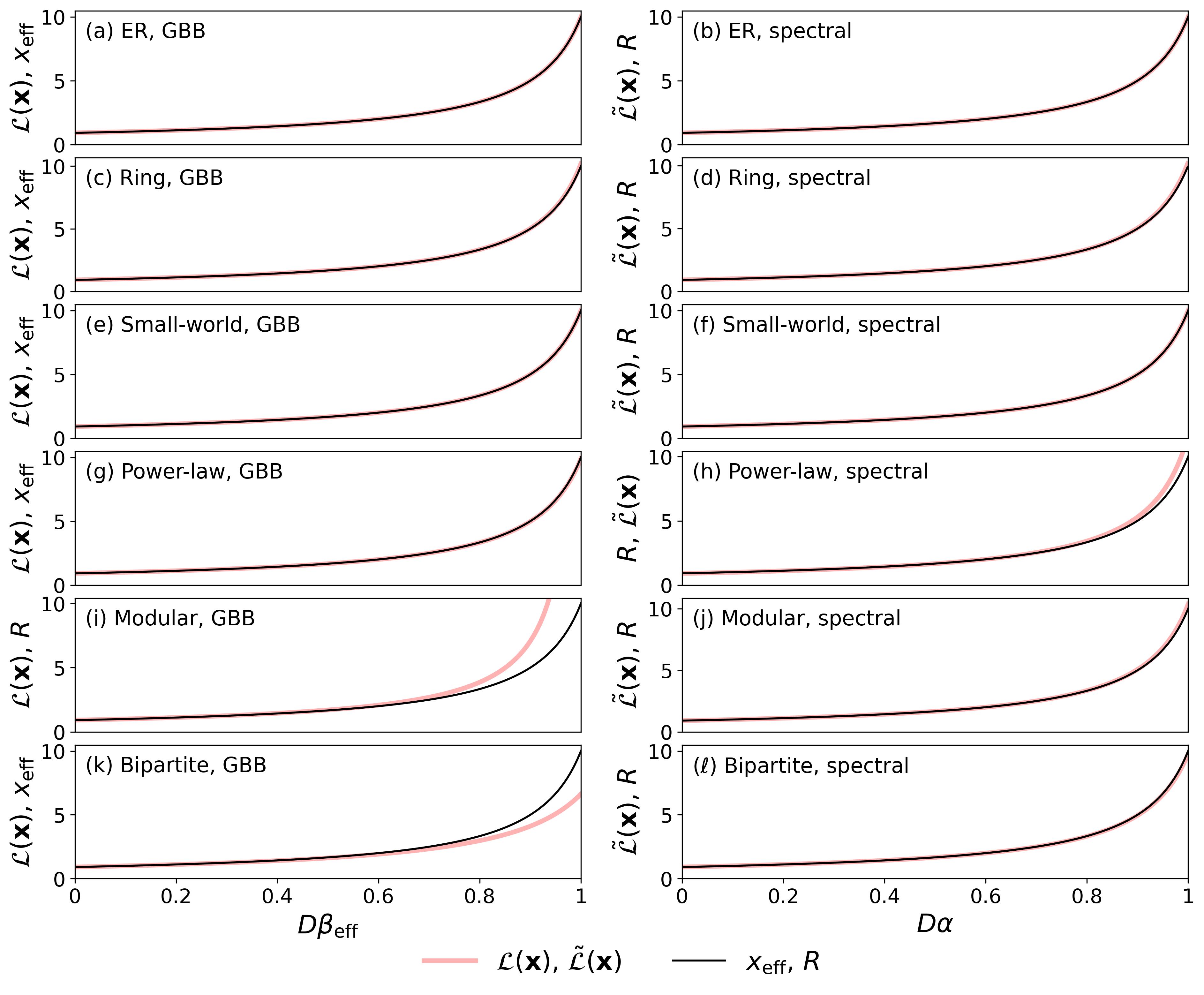}
		\caption{Accuracy of the two dimension reduction models for the GLV dynamics. See the caption of Fig.~\ref{fig_SIS_Model} for the legends.
		}
		\label{fig_GLV_Model}
	\end{figure}
	
	Figure~\ref{fig_Mutualistic_Model} shows the results for the mutualistic interaction dynamics. The results are similar to those for the coupled double-well dynamics in the sense that both GBB and spectral reductions are accurate for the ER, ring, and small-world graphons and not for the power-law, modular, and bipartite graphons. However, the approximation error is visibly larger for the ring graphon in the case of the present dynamics model than the coupled double-well dynamics. The accuracy of the GBB and spectral reductions across the six graphons is apparently similar.
	
	\clearpage
	\begin{figure}[h!]
		\centering
		\includegraphics[scale=0.48]{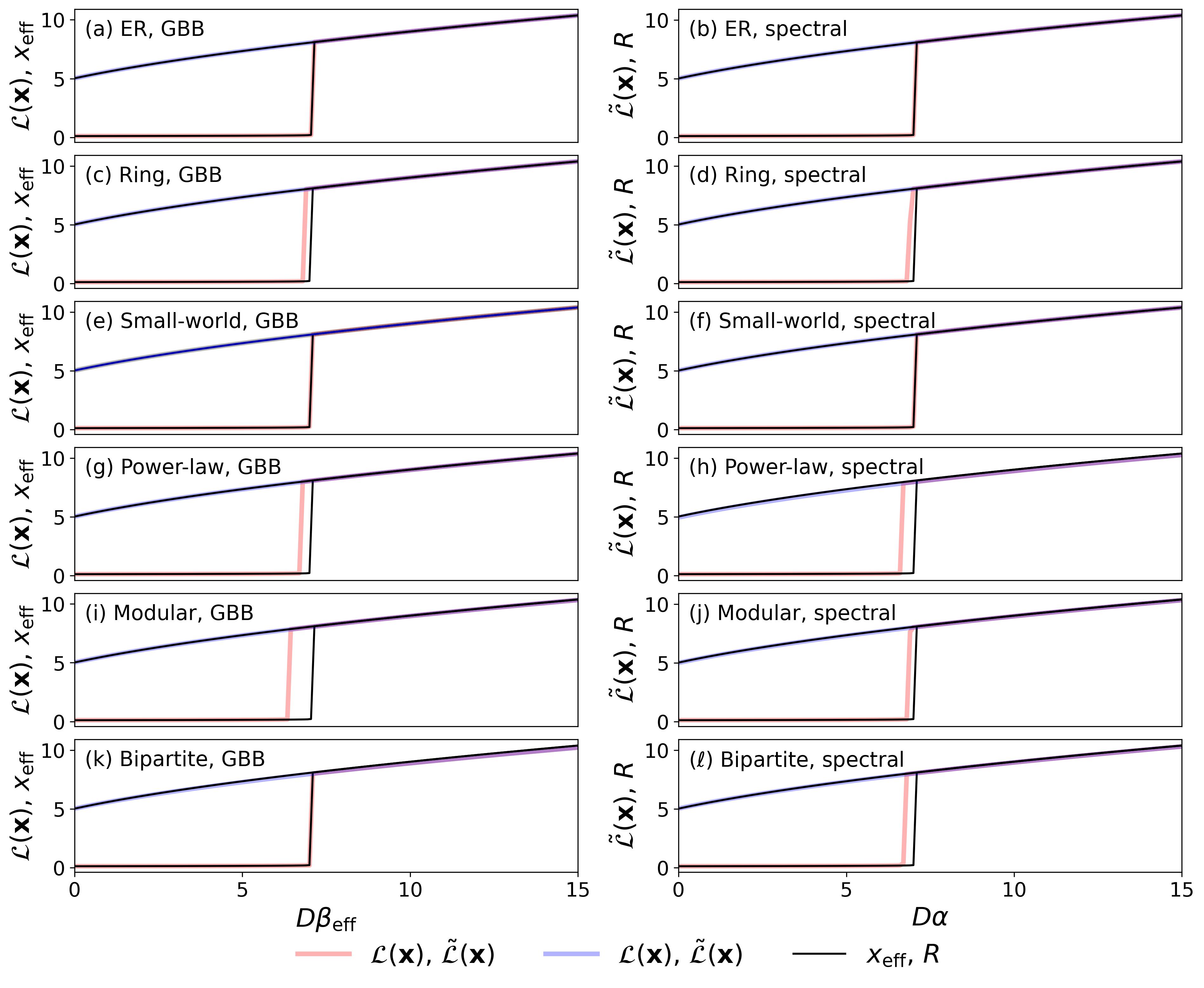}
		\caption{Accuracy of the two dimension reduction models for the mutualistic interaction dynamics. See the caption of Fig.~\ref{fig_DW_Model} for the legends.}
		\label{fig_Mutualistic_Model}
	\end{figure}
	
	Figure~\ref{fig_WC_Model} shows the results for the Wilson-Cowan model. The GBB reduction is reasonably accurate for all but the modular and bipartite graphons. In contrast, the spectral reduction is reasonably accurate for all but the power-law graphon.
	
	\clearpage
	\begin{figure}[h!]
		\centering
		\includegraphics[scale=0.48]{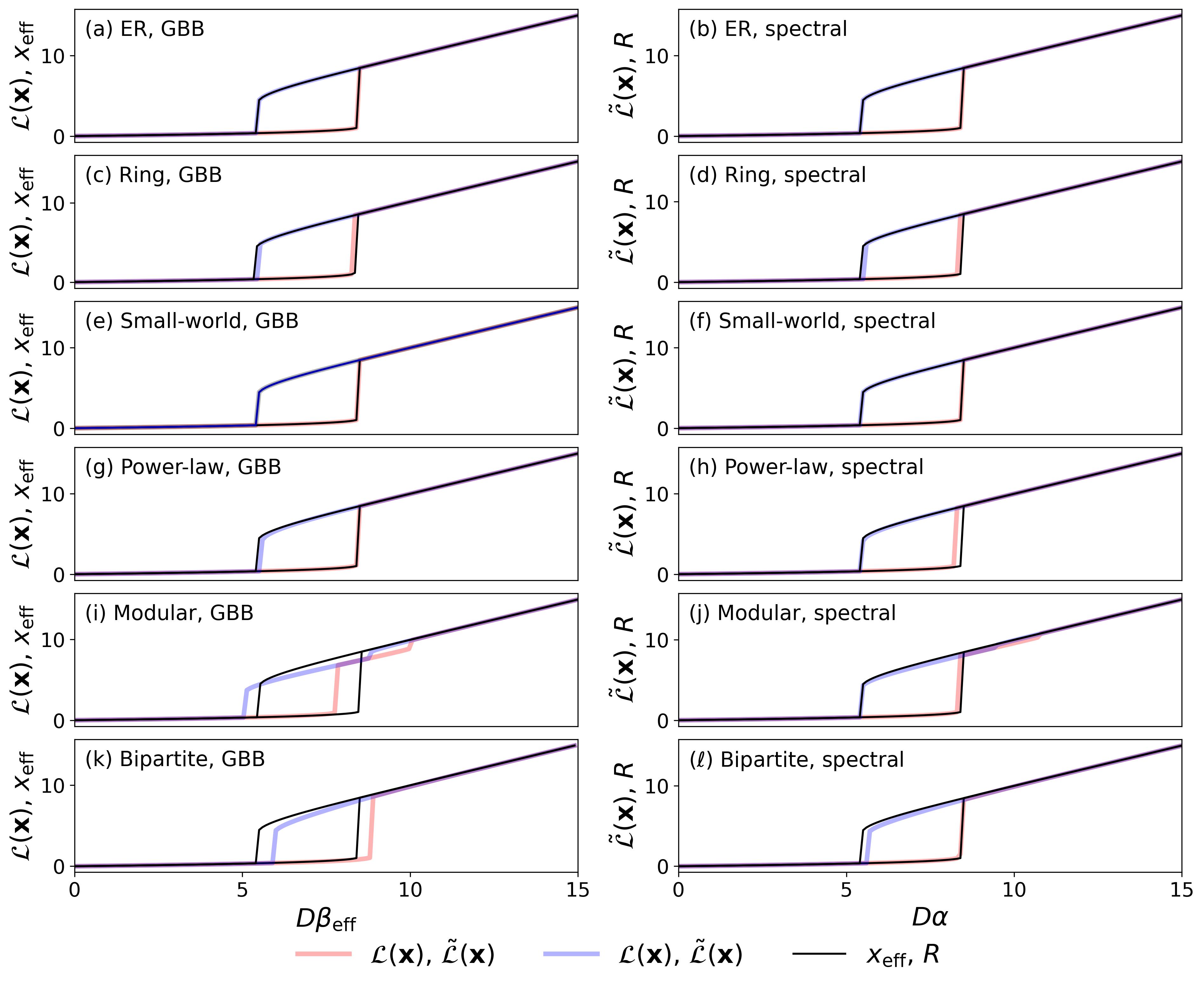}
		\caption{Accuracy of the two dimension reduction models for the Wilson-Cowan dynamics. See the caption of Fig.~\ref{fig_DW_Model} for the legends.
		}
		\label{fig_WC_Model}
	\end{figure}

	To summarize these results across the six dynamics models and six graphons, we show in Fig.~\ref{fig_RE_graphon} the relative error for each combination of dynamics model, graphon, and the initial condition if there are two (i.e., lower and upper) initial conditions considered for the chosen dynamics. Each panel represents a dynamics model. The horizontal and vertical axes represent the relative error for the GBB and spectral reductions, respectively. Figure~\ref{fig_RE_graphon} verifies that the ER, ring, and small-world graphons tend to yield small error across the dynamics and initial conditions. Another key observation is that, across the dynamics, graphons, and initial conditions, the GBB and spectral reductions are not particularly more accurate than each other, while one reduction method tends to be better than the other in some dynamics (see Fig.~\ref{fig_RE_graphon}(b) and (c)). 
	
	\clearpage
	\begin{figure}[h!]
		\centering
		\includegraphics[scale=0.35]{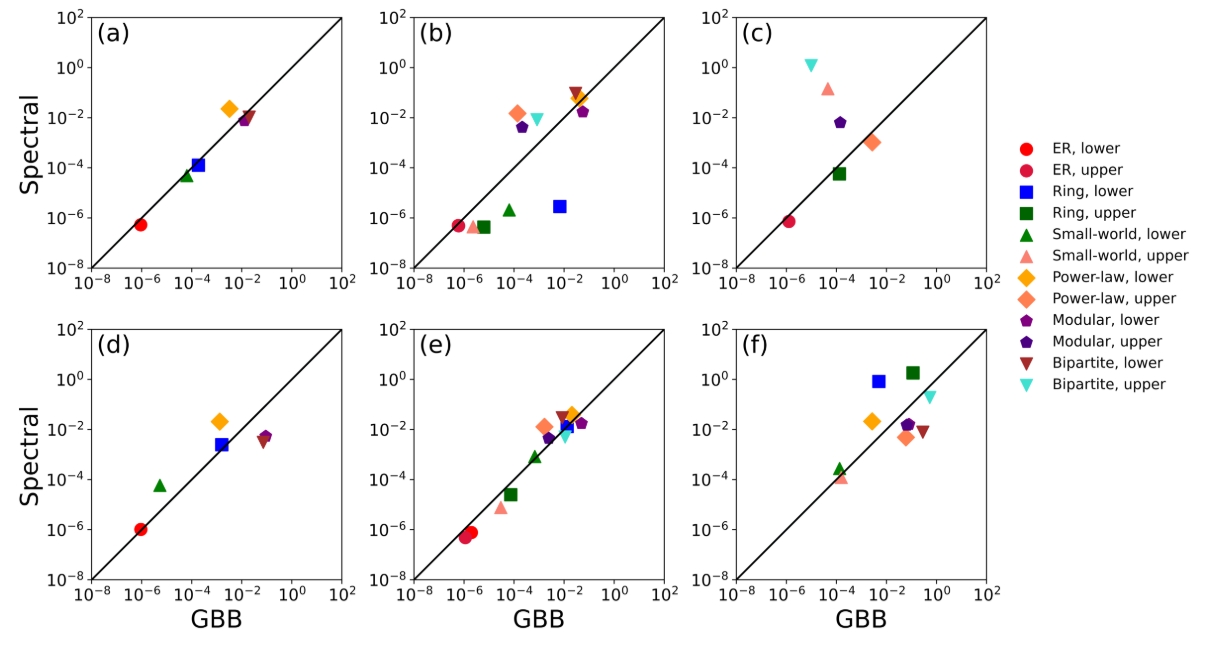}
		\caption{Comparison of the relative approximation error between the GBB and spectral reductions. (a) SIS. (b) Double-well. (c) Gene-regulatory. (d) GLV. (e) Mutualistic. (f) Wilson-Cowan. In the legend, `lower' and 'upper' represent the lower and upper initial conditions, respectively.}
		\label{fig_RE_graphon}
	\end{figure}
	
	\section{Conclusions}
	
	Motivated by various studies of dynamics on large networks, we have analyzed a class of nonlinear dynamical systems on infinitely large networks and its dimension reduction using graphons. Much of sections \ref{sec_2} and \ref{sec_3} is inspired by Ref.~\cite{medvedev2014nonlinear}, which dealt with a different class of dynamical systems on graphons. We provided different or simpler proofs wherever appropriate and pointed out the correspondence between the Galerkin approximation and the finite-network solution. We focused on dense graphons; developing similar mathematics for sparse graphons \cite{kaliuzhnyi2017semilinear, crucianelli2024interacting} warrants future work.
	
	Mathematically, we analyzed the accuracy of the GBB reduction by establishing an error bound. In numerical simulations, we used six dynamical system models and six types of graphons. We observed that the GBB and spectral reductions were reasonably accurate for most pairs of dynamics and graphon. The obtained accuracy for the graphon dynamics is comparable with that of these dimension reduction methods for finite networks examined in the literature~\cite{gao2016universal, laurence2019spectral, masuda2022dimension, kundu2022accuracy}. We also numerically found that the performance of the GBB and spectral reductions were similar on the six graphons. This last result is somewhat unexpected because the spectral reduction uses the global information about the network (i.e., the leading eigenvalue and eigenfunction of the graphon) and therefore more sophisticated than the GBB reduction, which only uses the weighted degree of each node. These results may be because the six graphons used in this study are benign, with simple $W$ functions having either high symmetry, monotonicity in terms of $y$ and $y'$, or being blockwise constant. To examine the accuracy of approximation in the case of more complicated graphons is saved for future work. 
	
	We did not attempt to mathematically show the convergence of the spectral reduction for finite networks to that for graphons. For the proof, we need the convergence of both eigenvalue and eigenvector of $A$ $(= W_N)$ to the eigenvalue and eigenfunction, respectively, of $W$.
	The convergence of the eigenvalue is known \cite{borgs2012convergent}. However, to the best of our knowledge, the convergence of the eigenvector to the graphon eigenfunction is generally not known. The convergence of the eigenvector to eigenfunction was shown in Theorem $3.6$ in \cite{gong2022regular}. However, certain conditions such as uniform convergence of the graph limit as well as the compact convergence of the Riemann sum to their respective integral, which general graphons $W$ may not satisfy, need to be met for this theorem to be applied. Justification of the spectral reduction for graphon dynamical systems warrants future work.
	
	\appendix
	\section{Proof of Lemma~\ref{lemma_1}\label{sec:proof-lemma-convergence}}
	
	Let $\xi_{N}(y, t) = x_{N}(y,t) - x(y,t)$. By subtracting \eqref{eq4} from \eqref{eq:pre-graphon-dynamical-system}, using \eqref{eq:Aij-from-W}, and taking the modulus, we obtain
	\begin{align}
		& \left|\frac{\partial \xi_{N}(y,t)}{\partial t}\right| \notag\\
		=&  \left|f(x_{N}(y,t)) - f(x(y,t)) + \int_{I}W_{N}(y,y')G(x_{N}(y,t), x_{N}(y',t))\mathrm{d}y' - \int_{I}W(y,y')G(x(y,t),x(y',t))\mathrm{d}y'\right| \notag\\
		=& \left|f(x_{N}(y,t)) - f(x(y,t)) + \int_{I}W_{N}(y,y')\left[G(x_{N}(y,t), x_{N}(y',t)) - G(x(y,t),x(y',t))\right]\mathrm{d}y' \right.\notag\\
		&\left.+ \int_{I}\left[W_{N}(y,y') - W(y,y')\right] G(x(y,t),x(y',t))\mathrm{d}y'\right|.
		\label{eq:convergence-proof-1}
	\end{align}
	By multiplying both sides of \eqref{eq:convergence-proof-1} with $\left|\xi_{N}(y,t)\right|$, integrating over $I$, and applying 
	Assumptions \ref{assumption_1}, we obtain
	\begin{align}\label{eq:sup_lemma_boundedness}
		\left| \frac{1}{2}\int_{I}\frac{\partial \xi^2_{N}(y,t)}{\partial t} \mathrm{d}y\right| \le& L\left(\int_{I}\left|x_{N}(y,t) - x(y,t)\right|\left|\xi_{N}(y,t)\right| \mathrm{d}y\right. \nonumber \\
		&\left.+ \int_{I^2} \left|W_{N}(y,y')\right|\left|(x_{N}(y,t), x_{N}(y',t))- (x(y,t),x(y',t))\right|\left|\xi_{N}(y,t)\right| \mathrm{d}y\mathrm{d}y'\right) \nonumber\\
		&+ \int_{I^2}\left|\left[W_{N}(y,y') - W(y,y')\right] G(x(y,t),x(y',t))\right| \left|\xi_{N}(y,t)\right|\mathrm{d}y\mathrm{d}y',
	\end{align}
	where we remind that $L = \max\{L_f, L_G\}$. Because $\left\|W_{N}\right\|_{L^\infty(I^2)} \leq 1$, we bound the second term on the right-hand side of \eqref{eq:sup_lemma_boundedness} as
	\begin{align}\label{eq:R2}
		& \int_{I^2} \left|W_{N}(y,y') \right| \left| (x_{N}(y,t), x_{N}(y',t))- (x(y,t),x(y',t)) \right|\left|\xi_{N}(y,t)\right| \mathrm{d}y\mathrm{d}y'\nonumber\\
		\leq& \int_{I^{2}}\left|\xi_{N}(y,t) - \xi_{N}(y',t) \right|\left|\xi_{N}(y,t)\right|\mathrm{d}y\mathrm{d}y'\nonumber\\
		\leq& 2 \left\|\xi_{N}(t)\right\|_{L^2(I^2)}^2.
	\end{align} 
	By applying the Cauchy-Schwartz inequality to the third term on the right-hand side of \eqref{eq:sup_lemma_boundedness}, we obtain
	\begin{align}\label{eq:R3}
		& \int_{I^2}\left|\left[W_{N}(y,y') - W(y,y')\right] G(x(y,t),x(y',t))\right|\left|\xi_{N}(y,t)\right|\mathrm{d}y\mathrm{d}y' \nonumber\\
		\leq& {\text{ess}\sup}_{(y,y',t)\in I^2 \times [0,T]}\left|G(x(y,t),x(y',t))\right|
		\int_{I^2}\left|W_{N}(y,y') - W(y,y')\right|\left|\xi_{N}(y,t)\right|\mathrm{d}y\mathrm{d}y' \nonumber\\
		\leq& C_{1}\left\|W_{N} - W\right\|_{L^2(I^2)}\left\|\xi_{N}(t)\right\|_{L^2(I)},
	\end{align}
	where $C_{1}= {\text{ess}\sup}_{(y,y',t)\in I^2 \times [0,T]}\left|G(x(y,t),x(y',t))\right|$.
	By substituting \eqref{eq:R2} and \eqref{eq:R3} in \eqref{eq:sup_lemma_boundedness}, we obtain
	\begin{align}
		\frac{1}{2}\frac{\partial \left\|\xi_{N}(t)\right\|_{L^2(I)}^2}{\partial t} &\leq L \left(\left\|\xi_{N}(t)\right\|_{L^2(I)}^2 + 2 \left\|\xi_{N}(t)\right\|_{L^2(I)}^2\right) + C_{1}\left\|W_{N} - W\right\|_{L^2(I^2)}\left\|\xi_{N}(t)\right\|_{L^2(I)},
	\end{align}
	that is,	    
	\begin{align}\label{eq:lemma}
		\frac{\partial \left\|\xi_{N}(t)\right\|^2}{\partial t} &\leq 6L \left\|\xi_{N}(t)\right\|_{L^2(I)}^2 + 2 C_{1}\left\|W_{N} - W\right\|_{L^2(I^2)}\left\|\xi_{N}(t)\right\|_{L^2(I)}.
	\end{align}
	Consider $\phi_{\epsilon}(t) := \sqrt{\left\|\xi_{N}(t)\right\|_{L^2(I)}^2 + \epsilon}$ where $\epsilon > 0$ is an arbitrary but fixed value. Then, \eqref{eq:lemma} implies that
	\begin{align}
		\frac{\partial \phi_{\epsilon}(t)^2}{\partial t} &\leq 6L \phi_{\epsilon}(t)^2 + 2 C_{1}\left\|W_{N} - W\right\|_{L^2(I^2)}\phi_{\epsilon}(t),
	\end{align}
	which reduces to
	\begin{align}
		\frac{\partial \phi_{\epsilon}(t)}{\partial t} &\leq 3L \phi_{\epsilon}(t) +  C_{1} \left\|W_{N} - W\right\|_{L^2(I^2)}.
		\label{eq:convergence-proof-before-Gronwall}
	\end{align}
	Applying the Gr$\ddot{\text{o}}$nwall's inequality to \eqref{eq:convergence-proof-before-Gronwall}, we obtain
	\begin{align}
		\sup_{t\in[0,T]} \phi_{\epsilon}(t) &\leq  \left( \phi_{\epsilon}(0) + C_{1}\left\|W_{N} - W\right\|_{L^2(I^2)}T \right) e^{3LT}.
	\end{align}
	Because $\epsilon>0$ is arbitrary, we obtain
	\begin{align}
		\left\|\xi_{N}(t)\right\|_{L^2(I)} & \leq \left(\left\|g_{N}-g\right\|_{L^2(I)} + C_{1}\left\|W_{N} - W\right\|_{L^2(I^2)}T \right) e^{3LT},
		\label{eq:convergence-lemma-statement-replicated}
	\end{align}
	which is equivalent to \eqref{eq:convergence-lemma-statement}.	
	
	\section{Proof of Theorem~\ref{thm:convergence}}\label{sub_Sxx}
	
	First, the hypothesis of the theorem implies that
	\begin{align}\label{eq:hypothesis}
		\left\|g_{N}-g\right\|_{L^2(I)} \to 0 ~\text{as}~ N\to \infty.
	\end{align}
	Second, because $W_{N}$ is the step graphon obtained from graphon $W$, we know that $W_{N}$ converges to $W$ as $N \to \infty$, i.e.,
	\begin{align}\label{eq:step_graphon_cgs}
		\left\|W_{N}-W\right\|_{L^2(I^2)} \to 0 ~\text{as}~ N\to \infty.
	\end{align}
	By applying \eqref{eq:hypothesis} and \eqref{eq:step_graphon_cgs} in \eqref{eq:convergence-lemma-statement}, we obtain 
	\begin{align}
		\left\|\pmb{x}_{N} - \pmb{x}\right\|_{C(L^2(I),[0,T])} & \to 0 ~\text{as}~ N \to \infty,
	\end{align}
	concluding the proof.	
	
	\section{Proof of Lemma~\ref{lemma_2} \label{sup_sec: cgs 1-dimen system}}
	
	Consider the one-dimensional equation in the discrete case, \eqref{eq_46}, and the one-dimensional equation in the continnum case given by
	\begin{equation}\label{eq_46**}
		\frac{dx_\text{eff}}{dt}=f(x_\text{eff})+\beta_\text{eff}\,G(x_\text{eff},x_\text{eff}).
	\end{equation}
	Let $\overline{\xi}_{N}(t) = x_{N\text{eff}}(t) - x_\text{eff}(t)$. By subtracting \eqref{eq_46**} from \eqref{eq_46} and taking the modulus, we obtain
	\begin{align}
		& \left|	\frac{d\overline{\xi}_{N}(t)}{dt}\right| \notag\\
		=& \left| f(x_{N\text{eff}}) - f(x_\text{eff}) + \beta_{N\text{eff}}\,G(x_{N\text{eff}},x_{N\text{eff}}) - \beta_\text{eff}\,G(x_\text{eff},x_\text{eff}) \right| \notag\\
		=&  \left| f(x_{N\text{eff}}) - f(x_\text{eff}) + \beta_{N\text{eff}}\,G(x_{N\text{eff}},x_{N\text{eff}}) - \beta_\text{eff}\,G(x_{N\text{eff}},x_{N\text{eff}}) + \beta_\text{eff}\,G(x_{N\text{eff}},x_{N\text{eff}}) - \beta_\text{eff}\,G(x_\text{eff},x_\text{eff}) \right| \notag\\%	
		\leq& \left|f(x_{N\text{eff}}) - f(x_\text{eff}) \right| + \left|\beta_{N\text{eff}} - \beta_\text{eff}\right|\left|G(x_{N\text{eff}},x_{N\text{eff}})\right| + \left|\beta_\text{eff}\right|\left|G(x_{N\text{eff}},x_{N\text{eff}})- G(x_\text{eff},x_\text{eff})\right|.
		\label{eq:GBB-convergence-proof-1}
	\end{align}
	By multiplying both sides of \eqref{eq:GBB-convergence-proof-1} with $\left|\overline{\xi}_{N}(t)\right|$, integrating over $[0,T]$, and applying Assumptions \ref{assumption_1}, we obtain
	\begin{align}
		\frac{1}{2}	\frac{d\left\|\overline{\xi}_{N}(t)\right\|^{2}_{L^2([0,T])}}{dt} &\leq L\left( \left\|\overline{\xi}_{N}(t)\right\|_{L^2([0,T])} + \left\|\beta_\text{eff}\right\|_{L^2([0,T])}\left\|\overline{\xi}_{N}(t)\right\|_{L^2([0,T])}\right)\nonumber\\
		+& C_{1}\left\|\beta_{N\text{eff}} - \beta_\text{eff}\right\|_{L^2([0,T])}\left\|\overline{\xi}_{N}(t)\right\|_{L^2([0,T])}.
	\end{align}
	We remind that $ L = \max\{L_{f}, L_{G}\}$ and that $C_{1}$ is an upper bound for ${\text{ess}\sup}_{t\in [0,T]} \left|G(x(t),x(t))\right|$ (see Appendix~\ref{sec:proof-lemma-convergence}).
	
	Next, we define $\overline{\phi}_{\epsilon}(t):= \sqrt{\left\|\overline{\xi}_{N}(t)\right\|^{2}_{L^2(I)}+ \epsilon}$, where $\epsilon>0$. By following the same procedure as in the proof of Lemma \ref{lemma_1}, we obtain
	\begin{equation}
		\left\|\overline{\xi}_{N}(t)\right\|_{L^2([0,T])} \leq \left(\left\|x_{N\text{eff}}(0) - x_\text{eff}(0)\right\|_{L^2([0,T])} + C_{1}\left\|\beta_{N\text{eff}} - \beta_\text{eff}\right\|_{L^2([0,T])}\right) e^{LT\left(1+ \beta_{\text{eff}}\right)}.
	\end{equation}
	By applying Remark \ref{rem:1} and \eqref{eq:1-dimen_ini}, we obtain \eqref{eq:1-dimen_cgs} as $N \to \infty$.
	
	\section{Derivation of the spectral reduction for graphon dynamical systems}\label{sec:derivation-SM-graphon}
	
	We consider a perturbation of the state $x(y,t)$ from the observable $R$ by setting $x(y,t)=R+\Delta x(y,t)$, where $\left|\Delta x(y,t)\right| \ll 1$.
	We obtain
	\begin{align}
		\dfrac{\partial R}{\partial t}
		%
		% &=\dfrac{\partial}{\partial t}\int_0^1a(x)x(x,t)\,dx\nonumber\\
		%
		&=\int_0^1a(y)\dfrac{\partial}{\partial t}x(y,t)\,dy\nonumber\\
		%
		% &=\int_0^1a(x)\left[ f(\omega(x,t))+\int_0^1W(x,y)\,G(\omega(x,t),\omega(y,t))\,dy \right]\,dx\nonumber\\
		%
		&=\int_0^1a(y) f(x(y,t))\,\mathrm{d}y + \int_0^1\int_0^1a(y)\,W(y,y')\,G(x(y,t),x(y',t))\,\mathrm{d}y'\,\mathrm{d}y.
	\end{align}
	The Taylor expansion yields
	\begin{align}
		\label{eq:spectral-method-Taylor-1}
		\int_0^1a(y) f(x(y,t))\,\mathrm{d}y&=\int_0^1a(y) f(R+\Delta x(y,t))\,\mathrm{d}y\nonumber\\
		&=\int_0^1a(y) [f(R)+Jf(R)\Delta x(y,t)+\mathcal{O}\left((\Delta x(y,t))^2\right)]\,\mathrm{d}y\nonumber\\
		&=f(R)+Jf(R) \int_0^1a(y) \Delta x(y,t)\,\mathrm{d}y+\int_0^1a(y) \mathcal{O}\left((\Delta x(y,t))^2\right)\,\mathrm{d}y,
	\end{align}
	where $Jf(R)$ is the Jacobian of $f(R)$, i.e., $Jf(R) = \frac{df}{dR}$.
	The second term on the right-hand side of \eqref{eq:spectral-method-Taylor-1} vanishes because
	\begin{align}
		\int_0^1a(y) \Delta x(y,t)\,\mathrm{d}y =& \int_0^1a(y) [x(y,t)-R]\,\mathrm{d}y \notag\\
		=& \int_{0}^{1}a(y) x(y,t) \mathrm{d}y-\int_{0}^{1}a(y)R \mathrm{d}y \notag\\
		=& R\left[1-\int_{0}^{1}a(y)\mathrm{d}y \right] = 0.
	\end{align}
	The Taylor expansion also yields
	\begin{align}
		& \int_0^1\int_0^1a(y)\,W(y,y')\,G(x(y,t),x(y',t))\,\mathrm{d}y'\,\mathrm{d}y \notag\\
		=& \int_0^1\int_0^1a(y)\,W(y,y')\,G(R+\Delta x(y,t),R+\Delta x(y',t))\,\mathrm{d}y'\,\mathrm{d}y \nonumber\\ 
		=& \int_0^1\int_0^1a(y)\,W(y,y')\, \left[G(R,R)+J_1G(R,R)\Delta x(y,t)\right.\\ \nonumber
		&\left.+J_2G(R,R)\Delta x(y',t)+\mathcal{O}\left((\Delta x(y,t))^2\right)+\mathcal{O}\left((\Delta x(y',t))^2\right) \right] \,\mathrm{d}y'\,\mathrm{d}y \nonumber\\ 
		=& G(R,R)\int_0^1\int_0^1a(y)\,W(y,y') \,\mathrm{d}y'\,\mathrm{d}y \nonumber\\  
		&+J_1G(R,R)\int_0^1\int_0^1a(y)\,W(y,y')\, \Delta x(y,t) \,\mathrm{d}y'\,\mathrm{d}y \nonumber\\
		&+J_2G(R,R)\int_0^1\int_0^1a(y)\,W(y,y')\, \Delta x(y',t) \,\mathrm{d}y'\,\mathrm{d}y \nonumber\\
		&+\int_0^1\int_0^1a(y)\,W(y,y')\, \left[\mathcal{O}\left((\Delta x(y,t))^2\right)+\mathcal{O}\left((\Delta x(y',t))^2\right)\right] \,\mathrm{d}y'\,\mathrm{d}y,
	\end{align}
	where $J_1$ and $J_2$ is the partial derivative of $G$ with respect to the first and second argument, respectively.
	Furthermore, we obtain
	\begin{align}
		\int_0^1\int_0^1a(y)\,W(y,y')\, \Delta x(y,t) \,\mathrm{d}y'\,\mathrm{d}y&=\int_0^1\int_0^1a(y)\,W(y,y')\, \left[x(y,t)-R\right] \,\mathrm{d}y'\,\mathrm{d}y\nonumber\\
		&=\int_0^1\int_0^1a(y)\,W(y,y')\,x(y,t)\,\mathrm{d}y'\,\mathrm{d}y-R\int_0^1\int_0^1a(y)\,W(y,y')\, \mathrm{d}y'\,\mathrm{d}y\nonumber\\
		&=\int_0^1a(y)\,k(y)\,x(y,t)\,\mathrm{d}y-R\,\alpha,
	\end{align}
	where $\alpha=\int_0^1\int_0^1a(y)\,W(y,y') \,\mathrm{d}y'\,\mathrm{d}y$, and
	\begin{align}
		\int_0^1\int_0^1a(y)\,W(y,y')\, \Delta x(y',t) \,\mathrm{d}y'\,\mathrm{d}y&=\int_0^1\int_0^1a(y)\,W(y,y')\, \left[x(y',t)-R\right] \,\mathrm{d}y'\,\mathrm{d}y \nonumber\\ 
		&=\int_0^1\int_0^1a(y)\,W(y,y')\,x(y',t)\,\mathrm{d}y'\,\mathrm{d}y-R\int_0^1\int_0^1a(y)\,W(y,y')\, \mathrm{d}y'\,\mathrm{d}y \nonumber\\
		&=\int_0^1\int_0^1a(y)\,W(y,y')\,x(y',t)\,\mathrm{d}y'\,\mathrm{d}y-R\,\alpha.
	\end{align}
	Therefore, we obtain
	\begin{align}\label{eq_SM-taylor}
		\dfrac{\partial R}{\partial t}=&f(R)+G(R,R)\,\alpha \nonumber\\ 
		&+J_1G(R,R)\left[\int_0^1a(y)\,k(y)\,x(y,t)\,\mathrm{d}y-R\,\alpha\right]
		+J_2G(R,R)\left[\int_0^1\int_0^1a(y)\,W(y,y')\,x(y',t)\,\mathrm{d}y'\,\mathrm{d}y-R\,\alpha\right] \nonumber\\ 
		&+\int_0^1\int_0^1a(y)\,W(y,y')\, \left[\mathcal{O}\left((\Delta x(y,t))^2\right)+\mathcal{O}\left((\Delta x(y',t))^2\right)\right] \,\mathrm{d}y'\,\mathrm{d}y.
	\end{align}
	
	For the first-order quantities in terms of $\Delta x(y,t)$ to vanish, the $J_1 G$ and $J_2 G$ terms need to vanish. Therefore, we require
	\begin{align}
		0 =& \int_0^1a(y)\,k(y)\,x(y,t)\,\mathrm{d}y-R\,\alpha \notag\\
		= & \int_0^1 x(y,t)\left[k(y)\,a(y)-\alpha a(y)\right]\,\mathrm{d}y \notag\\
		= & \int_0^1 x(y,t)\left[\int_0^1W(y,y')\,a(y)\,\mathrm{d}y'-\alpha a(y)\right]\,\mathrm{d}y
	\end{align}
	and
	\begin{align}
		0 =& \int_0^1\int_0^1a(y)\,W(y,y')\,x(y',t)\,\mathrm{d}y'\,\mathrm{d}y-R\,\alpha \notag\\
		=& \int_0^1 x(y',t)\left[\int_0^1W(y,y')\,a(y)\,\mathrm{d}y-\alpha a(y')\right]\,\mathrm{d}y'.
	\end{align}
	In other words, if we impose
	\begin{equation}
		\int_0^1W(y,y')\,a(y)\,\mathrm{d}y'=\alpha\,a(y)
		\label{eq:SM-graphon-condition-1}
	\end{equation}
	and 
	\begin{equation}
		\int_0^1W(y,y')\,a(y)\,\mathrm{d}y=\alpha\,a(y'),
		\label{eq:SM-graphon-condition-2}
	\end{equation}
	we obtain
	\begin{equation}\label{eq_49}
		\dfrac{d R}{d t}=f(R)+\alpha\,G(R,R),
	\end{equation}
	i.e., \eqref{eq:SM-original}.
	
	Because \eqref{eq:SM-graphon-condition-1} and \eqref{eq:SM-graphon-condition-2} cannot be simultaneously satisfied in general, we only impose \eqref{eq:SM-graphon-condition-2}. We also neglect the second-order terms in \eqref{eq_SM-taylor}.
	This equation implies that $\alpha$ and $a(y)$ are a paired eigenvalue and eigenfunction, respectively, of graphon $W$. Analyses for finite networks suggest that it is practically good to use the leading eigenvalue and eigenvector for the spectral reduction \cite{laurence2019spectral, thibeault2020threefold, masuda2022dimension}. Therefore, we use the leading eigenvalue and eigenvector pair of the graphon $W$ for the spectral reduction.
	
	\section*{Acknowledgments}
	
	N.M. is supported in part by the NSF under Grant No.\,DMS-2204936, in part by JSPS KAKENHI under Grants No.\,JP21H04595, No.\,23H03414, No.\,24K14840, and No.\,24K030130, and in part by Japan Science and Technology Agency (JST) under Grant No.\,JPMJMS2021. During the preparation of this manuscript, the authors used  U-M (University of Michigan) GPT 5.5 for language editing and assisting mathematical analyses. 
	\bibliographystyle{unsrt}
	\bibliography{references}
\end{document}